\documentclass[10pt]{amsart}

\usepackage{amsfonts}
\usepackage{amsmath}
\usepackage{amssymb}
\usepackage{amsthm}
\usepackage{enumerate}
\usepackage{mdwlist}
\usepackage{mathrsfs}
\usepackage{graphicx}
\usepackage{latexsym}
\usepackage{stmaryrd}
\usepackage{verbatim}
\usepackage{xcolor}
\usepackage{comment}
\usepackage{url}
\usepackage[foot]{amsaddr}

\newtheorem{theorem}{Theorem}[section]
\newtheorem{proposition}[theorem]{Proposition}
\newtheorem{lemma}[theorem]{Lemma}
\newtheorem{corollary}[theorem]{Corollary}

\theoremstyle{remark}
\newtheorem{definition}[theorem]{Definition}
\newtheorem{remark}[theorem]{Remark}
\newtheorem{example}[theorem]{Example}

\newcommand{\NN}{\mathbb{N}}
\newcommand{\RR}{\mathbb{R}}
\newcommand{\PP}{\mathbb{P}}
\newcommand{\EE}{\mathbb{E}}
\newcommand{\fluc}[3]{J_{{#1},{#2}}\{{#3}\}}
\newcommand{\upcr}[3]{U_{{#1},{#2}}\{{#3}\}}
\newcommand{\dcr}[3]{D_{{#1},{#2}}\{{#3}\}}
\newcommand{\dcrinf}[2]{D_{{#1}}\{{#2}\}}
\newcommand{\crs}[3]{C_{{#1},{#2}}\{{#3}\}}
\newcommand{\flucinf}[2]{J_{{#1}}\{{#2}\}}
\newcommand{\upcrinf}[2]{U_{{#1}}\{{#2}\}}
\newcommand{\crsinf}[2]{C_{{#1}}\{{#2}\}}

\newcommand{\norm}[1]{\lVert {#1} \rVert}

\newcommand{\seq}[1]{\{{#1}\}}

\begin{document}

\title[On quantitative convergence for stochastic processes]{On quantitative convergence for stochastic processes: Crossings, fluctuations and martingales}

\author{Morenikeji Neri and Thomas Powell}
\address{Department of Computer Science, University of Bath}
\email{$\{$mn728,trjp20$\}$@bath.ac.uk}

\date{\today}

\begin{abstract}
We develop a general framework for extracting highly uniform bounds on local stability for stochastic processes in terms of information on fluctuations or crossings. This includes a large class of martingales: As a corollary of our main abstract result, we obtain a quantitative version of Doob's convergence theorem for $L_1$-sub- and supermartingales, but more importantly, demonstrate that our framework readily extends to more complex stochastic processes such as almost-supermartingales, thus paving the way for future applications in stochastic optimization. Fundamental to our approach is the use of ideas from logic, particularly a careful analysis of the quantifier structure of probabilistic statements and the introduction of a number of abstract notions that represent stochastic convergence in a quantitative manner. In this sense, our work falls under the `proof mining' program, and indeed, our quantitative results provide new examples of the phenomenon, recently made precise by the first author and Pischke, that many proofs in probability theory are proof-theoretically tame, and amenable to the extraction of quantitative data that is both of low complexity and independent of the underlying probability space.\medskip
\end{abstract}

\maketitle

%%%%%%%%%%%%%%%%%%%%%%%%%%%%%%%%%%%%%%%%%%%%%%%%%%%%%%%%%%%%%%%%%%%%%%%%%%%%%%%%%%%%%%%%%%
%%%%%%%%%%%%%%%%%%%%%%%%%%%%%%%%%%%%%%%%%%%%%%%%%%%%%%%%%%%%%%%%%%%%%%%%%%%%%%%%%%%%%%%%%%
\section{Introduction}
\label{sec:intro}
%%%%%%%%%%%%%%%%%%%%%%%%%%%%%%%%%%%%%%%%%%%%%%%%%%%%%%%%%%%%%%%%%%%%%%%%%%%%%%%%%%%%%%%%%%
%%%%%%%%%%%%%%%%%%%%%%%%%%%%%%%%%%%%%%%%%%%%%%%%%%%%%%%%%%%%%%%%%%%%%%%%%%%%%%%%%%%%%%%%%%

Applied proof theory (or \emph{proof mining}) \cite{kohlenbach:08:book} is a subfield of logic in which proofs of mathematical theorems are carefully analysed with the aim of strengthening those theorems, typically through obtaining quantitative information or showing that they hold in a generalised, more abstract setting. Proof mining has achieved considerable success over the last decades, primarily in areas related to nonlinear analysis, such as fixed point theory and convex optimization. This paper is part of a new approach to expand proof mining into probability theory. We aim to set the groundwork for new applications involving stochastic processes, particularly stochastic optimization, by providing the first quantitative study of martingales from this perspective. More specifically, we provide quantitative versions of some classic martingale convergence theorems based on locating regions of \emph{local stability} as in \cite{avigad-gerhardy-towsner:10:local}, which represent instances of a flexible, general method for obtaining bounds on local stability from quantitative information on fluctuations or upcrossings.   

%%%%%%%%%%%%%%%%%%%%%%%%%%%%%%%%%%%%%%%%%%%%%%%%%%%%%%%%%%%%%%%%%%%%%%%%%%%%%%%%%%%%%%%%%%
\subsection{Background: Proof mining and probability}
\label{sec:intro:background}
%%%%%%%%%%%%%%%%%%%%%%%%%%%%%%%%%%%%%%%%%%%%%%%%%%%%%%%%%%%%%%%%%%%%%%%%%%%%%%%%%%%%%%%%%%

Originating in Kreisel's ``unwinding'' program of the 1950s \cite{kreisel:51:proofinterpretation:part1,kreisel:52:proofinterpretation:part2}, and brought to maturity from the 1990s onwards by Kohlenbach and his collaborators, proof mining utilises ideas and techniques from logic to produce both quantitative and qualitative improvements of existing theorems. Progress in this area combines concrete applications, which usually involve the extraction of highly uniform computational information from seemingly nonconstructive proofs while at the same time showing that these proofs can be adapted to a more general setting, with so-called \emph{logical metatheorems}: theorems from logic that identify those applications as instances of general proof-theoretic phenomena, explaining their success and guiding the search for future applications.

A typical result here is the extraction, from a proof that some sequence converges, of an explicit rate of convergence. In situations where computable rates of convergence do not in general exist (a common occurrence first observed by Specker \cite{specker:49:sequence} and also present in martingale convergence, as we discuss in Remark \ref{rem:specker}), one can usually obtain so-called rates of \emph{metastability}, which informally speaking provide bounds on how far we need to look to find large intervals of local stability. Such metastable convergence theorems have been independently studied outside of logic by e.g. Tao (see \cite{tao:08:ergodic} or the essay \cite{tao:07:softanalysis}), where they can be naturally viewed as ``finitisations'' of infinitary convergence statements. One of the simplest examples of this phenomenon is the following metastable version of the monotone convergence principle:

\begin{theorem}[Finitary monotone convergence, cf. \cite{kohlenbach:08:book,tao:07:softanalysis}]
\label{res:finitary:monotone:convergence}
    Given $\varepsilon,K>0$ and $g:\NN\to\NN$ there exists some $N$ dependent only on these parameters such that whenever $\seq{x_n}$ is a monotone sequence in $[-K,K]$, there exists some $n\leq N$ such that $|x_i-x_j|<\varepsilon$ for all $n\leq i\leq j\leq n+g(n)$. Moreover, we can assign to $N$ the following concrete value:
    \[
    N:=\tilde g^{(\lceil 2K/\varepsilon\rceil)}(0)
    \]
    for $\tilde g(n):=n+g(n)$, where $\tilde g^{(i)}$ denotes the $i$th iteration of $\tilde g$.
\end{theorem}

To date, proof mining has focused primarily on nonlinear analysis and (nonstochastic) optimization, as documented in the recent survey papers \cite{kohlenbach:17:recent,kohlenbach:19:nonlinear:icm}, where the extraction of highly uniform rates of convergence or metastability is a standard result. Here, many convergence proofs make use of the monotone convergence principle in some way, and the resulting rates of metastability are very often of a similar form to those of Theorem \ref{res:finitary:monotone:convergence}, involving a simple iteration of the function, where the number of iterates is intuitively connected to how many $\varepsilon$-fluctuations the underlying sequence can experience. Fluctuations and related concepts play a fundamental role in the convergence theory of martingales, and it is therefore natural to ask whether similarly elegant finitisations are possible for stochastic convergence. 

Probability theory, in general, represents a tantalising and hitherto largely unexplored area of application for proof mining, where a particular challenge is posed by the subtlety of convergence in the probabilistic setting and the fact that even the simplest proofs typically make use of highly nonconstructive techniques. Work in this direction is largely restricted to some important case studies of Avigad et al. in ergodic theory and measure theory \cite{avigad-dean-rute:12:dominated,avigad-gerhardy-towsner:10:local,avigad-rute:14:oscillation}, along with Arthan and Oliva's study of the Borel-Cantelli lemma \cite{arthan-oliva:21:borel-cantelli}. More recently, the first author, together with Pischke, has established a first general metatheorem for probability theory \cite{neri-pischke:pp:formal}, developing a proof-theoretically tame system for reasoning about probability spaces and explaining the success of the aforementioned case studies, while reinforcing the idea that probability represents a promising area of application for proof-theoretic methods. These works set the stage perfectly for a dedicated quantitative study of stochastic processes.

%%%%%%%%%%%%%%%%%%%%%%%%%%%%%%%%%%%%%%%%%%%%%%%%%%%%%%%%%%%%%%%%%%%%%%%%%%%%%%%%%%%%%%%%%%
\subsection{Main contributions}
\label{sec:intro:contributions}
%%%%%%%%%%%%%%%%%%%%%%%%%%%%%%%%%%%%%%%%%%%%%%%%%%%%%%%%%%%%%%%%%%%%%%%%%%%%%%%%%%%%%%%%%%

The purpose of this paper is to lay down some foundations towards future applications of proof mining in stochastic optimization by providing a quantitative study of martingales, one of the central concepts in this area. Here, important quantitative information is already provided through fluctuation bounds and upcrossing inequalities, where the latter form a key ingredient in establishing classic martingale convergence theorems \cite{doob:53:stochastic}. Fluctuations and upcrossings have been explored in depth in the context of both probability theory and ergodic theory (see, for example, \cite{hochman:09:upcrossing,Ivanov:oscillations:96,jones-ostrovskii-rosenblatt:96:square,kachurovskii:96:convergence,kachurovskii2007general,kalikow1999fluctuations,warren:21:fluctuations}), and from the perspective of logic in e.g. \cite{avigad-gerhardy-towsner:10:local,avigad-rute:14:oscillation,kohlenbach-safarik:14:fluctuations}. The main technical results of this paper focus on the route from fluctuations or upcrossing bounds to metastable convergence rates, in other words, how one can use such bounds to locate regions of local stability for stochastic processes.

We can break down our contributions into two main parts. On the more foundational side, by exploiting the monotonicity properties of the quantifier structure of various deterministic notions, we arrive at a number of abstract definitions for bounding fluctuations and crossings, along with general theorems that relate these to various notions of metastable convergence. In this probabilistic setting, we consider two distinct forms of metastable convergence already studied in \cite{avigad-dean-rute:12:dominated} -- pointwise and uniform -- though crucially, we consider simplified `learnable' versions that focus on iterations of the function $g$.  In this way, we are able to clearly relate these notions of metastable convergence with bounds on fluctuations and upcrossings that arise naturally in the quantitative study of convergence for stochastic processes and ergodic averages. Through our abstract approach we obtain generalisations of results of Kachurovskii \cite{kachurovskii:96:convergence} and, more recently, of Avigad et al.\ \cite{avigad-gerhardy-towsner:10:local} that apply to arbitrary stochastic processes under very general and abstract assumptions.

Our second main contribution consists of providing specific rates of metastability for a number of important stochastic convergence theorems. As a relatively simple example, we prove the following stochastic analogue of Theorem \ref{res:finitary:monotone:convergence}, which forms a finitary version of Doob's convergence theorem for $L_1$-bounded sub- and supermartingales:

\begin{theorem}
    [Finitary martingale convergence]
    \label{res:finitary:supermartingale}
    Given $\varepsilon,\lambda,K>0$ and $g:\NN\to\NN$ there exists some $N$ dependent only on these parameters such that whenever $\seq{X_n}$ is a sub- or supermartingale with 
    \[
    \sup_{n\in\NN}\EE|X_n|<K
    \]
    there exists some $n\leq N$ such that
    \[
    |X_i-X_j|<\varepsilon \ \ \ \mbox{for all $n\leq i\leq j\leq n+g(n)$}
    \]
    with probability $>1-\lambda$. Moreover, we can assign to $N$ the following concrete value:
    \[
    N:=\tilde g^{(\lceil c(K/\lambda\varepsilon)^2\rceil)}(0)
    \]
    for $\tilde g(n):=n+g(n)$ and suitable constant $c>0$ that can be calculated explicity.
\end{theorem}

Our main abstract quantitative theorem allows us to prove explicit bounds $N$ on the local stability of stochastic processes under a range of different assumptions (where, in most cases, we show that our bounds are optimal in a certain sense). Most importantly, we give an example of how our abstract results apply more generally to almost-supermartingales with error terms, looking ahead to further applications in stochastic optimization. Here, numerous convergence theorems are proven by reducing stochastic algorithms to almost-supermartingales for which convergence can be established, the most well known result of this kind being the Robbins-Siegmund theorem \cite{robbins-siegmund:71:lemma} (see \cite{franci-grammatico:convergence:survey:22} for a recent survey of how martingales convergence in the form of the Robbins-Siegmund theorem is used in everything from stochastic gradient descent methods to computing Nash-equilibria). Not only is our analysis of the Doob martingale convergence theorem essential for providing a quantitative version of more sophisticated results on almost-supermartingales but our abstract treatment of fluctuation and upcrossing bounds is aimed to help facilitate its easy extension to such results, as illustrated in our Theorem \ref{res:optimization}. In all cases, our bounds are highly uniform in the sense that they are independent of both the underlying probability space and the precise distribution of the random variables (typically, the only information required is a bound on their mean). That such uniform quantitative information can be generally obtained from proofs in probability theory is explained in \cite{neri-pischke:pp:formal} and further discussed in Section \ref{sec:future:prooftheory} below.

%%%%%%%%%%%%%%%%%%%%%%%%%%%%%%%%%%%%%%%%%%%%%%%%%%%%%%%%%%%%%%%%%%%%%%%%%%%%%%%%%%%%%%%%%%
\subsection{This paper as part of a wider project}
\label{sec:intro:project}
%%%%%%%%%%%%%%%%%%%%%%%%%%%%%%%%%%%%%%%%%%%%%%%%%%%%%%%%%%%%%%%%%%%%%%%%%%%%%%%%%%%%%%%%%%

We view the results presented here as belonging to a broader effort by the authors and their collaborators to provide a major advance of proof mining into probability theory, incorporating both new quantitative theorems (as in this paper) together with logical results (as in \cite{neri-pischke:pp:formal}) that seek to devise `proof-theoretically tame' logical systems for probability together with metatheorems that guarantee the extractability of quantitative information from proofs formalisable within those systems. For this reason, we conclude the paper with an extended discussion outlining the importance of our results within this broader context, both how they provide insight into proof-theoretic aspects of stochastic processes and open the door towards more advanced applications in stochastic optimization. 

%%%%%%%%%%%%%%%%%%%%%%%%%%%%%%%%%%%%%%%%%%%%%%%%%%%%%%%%%%%%%%%%%%%%%%%%%%%%%%%%%%%%%%%%%%
\subsection{Related work}
\label{sec:intro:related}
%%%%%%%%%%%%%%%%%%%%%%%%%%%%%%%%%%%%%%%%%%%%%%%%%%%%%%%%%%%%%%%%%%%%%%%%%%%%%%%%%%%%%%%%%%

Many quantitative results on martingales and closely related notions from measure theory such as ergodic averages appear in the mainstream (non-logic) literature. Important examples of this include the detailed study by Kachurovskii \cite{kachurovskii:96:convergence}, along with the work of Jones et al. (e.g. \cite{jones:1998:oscillation,jones:1999:counting}) -- which have been further developed in case studies in proof-mining (see in particular \cite{avigad-rute:14:oscillation}). Quantitative results along these lines continue to be developed (e.g. \cite{jones:etal:08:variational,kachurovskii-podvigin:16:ergodic,tao:etal:10:carleson}). These often take the form of variational inequalities or bounds on oscillations (i.e. fluctuations), and particularly for convergence in the $L_2$-norm, rates of metastability can sometimes be directly inferred from these. On the other hand, our focus is on metastable formulations of \emph{almost sure convergence}, and these seem much more fundamentally connected to upcrossing inequalities, which have received less attention. Nevertheless, throughout the paper, we take care to highlight how our abstract theorems connect with known quantitative results, which, among other things, help us provide examples and optimality arguments.

In terms of related work in logic outside of proof mining, it should be noted that martingales have been extensively studied from the perspective of computability theory and algorithmic randomness, where rates of convergence also play a central role: See in particular \cite{hoyrup-rute:21:algorithmic:randomness,rute:etal:algorithmic:doob} and also the preprint \cite{rute:martingale:draft}. However, the focus is quite different: There, the aim is to impose additional computability structure on martingales, primarily with the aim of achieving computable rates of convergence, whereas we assume nothing at all about the underlying probability space or random variables in order to obtain quantitative information that is independent of their structure. In that way, we do not study the computability theory of probabilistic convergence as such, but instead aim to provide uniform finitisations of nonconstructive convergence theorems. 

It should also be noted that metastability in ergodic theory has been studied from the model theoretic perspective via ultraproducts \cite{avigad-iovino:13:ultraproducts,goldbring-towsner:14:measures}. Again, this approach is very different from ours, for example we produce concrete bounds for metastable convergence theorems, whereas the techniques of e.g. \cite{avigad-iovino:13:ultraproducts} only establish the \emph{existence} of uniform rates of metastability. Nevertheless, in the future it would be fascinating to explore whether a concrete connection can be found between the \emph{proof theoretic} approach to probabilistic convergence exemplified here and in \cite{neri-pischke:pp:formal}, and the very different world of model theory and ultraproducts\footnote{The authors are grateful to Henry Towsner for suggesting that there may be interesting parallels between these two approaches to metastability.}.

%%%%%%%%%%%%%%%%%%%%%%%%%%%%%%%%%%%%%%%%%%%%%%%%%%%%%%%%%%%%%%%%%%%%%%%%%%%%%%%%%%%%%%%%%%
\subsection{Organisation}
\label{sec:intro:organisation}
%%%%%%%%%%%%%%%%%%%%%%%%%%%%%%%%%%%%%%%%%%%%%%%%%%%%%%%%%%%%%%%%%%%%%%%%%%%%%%%%%%%%%%%%%%

Broadly speaking, this paper consists of four main parts, each with a somewhat different emphasis. In Sections \ref{sec:deterministic}-\ref{sec:moduli}, we give an abstract account of several notions of probabilistic convergence and arrive at a number of important definitions that arise from logical considerations, specifically the quantifier structure of almost sure statements. Sections \ref{sec:pointwise} and \ref{sec:uniform} then set out our main abstract quantitative results, which relate crossings to (learnable) metastable rates in two different ways. The longer Section \ref{sec:applications} then provides a series of concrete applications, primarily in martingale theory. Finally, in Section \ref{sec:future}, we give a more informal discussion on the underlying logical phenomena at play, along with open questions and future work. We stress that no knowledge of logic is required to understand the main results of the paper: Mathematical logic is restricted to Section \ref{sec:future:prooftheory}, which is included for the benefit of the more foundationally inclined reader.

\begin{comment}
We begin, in Section \ref{sec:deterministic}, by recalling a number of elementary results from proof mining concerning ordinary monotone bounded sequences, formulating these in terms of upcrossings and fluctuations and reobtaining the well-known metastable rates in this case. The purpose of this section is to motivate our main results and facilitate a comparison with the (much more intricate) stochastic setting. We follow this in Section \ref{sec:finitizing} by presenting a number of abstract logical and quantitative notions, including all of the main moduli that will be used to formulate our abstract results, before treating the pointwise (Sections \ref{sec:fluctuations} and \ref{sec:pointwise}) and uniform (Section \ref{sec:stochlearnability} and \ref{sec:crossing}) cases separately. Basic applications for martingales convergence and the ergodic theorem are given in Section \ref{sec:applications}, followed by our more involved example for almost-supermartingales in Section \ref{sec:optimization}. Up to this point, everything is presented in standard mathematical language with minimal reference to the underlying proof theory: In Section \ref{sec:prooftheory} we give some logical insights into why our uniform rates are possible, before concluding in Section \ref{sec:further} with a detailed overview of future work, both theoretical and applied.
\end{comment}

%%%%%%%%%%%%%%%%%%%%%%%%%%%%%%%%%%%%%%%%%%%%%%%%%%%%%%%%%%%%%%%%%%%%%%%%%%%%%%%%%%%%%%%%%%
%%%%%%%%%%%%%%%%%%%%%%%%%%%%%%%%%%%%%%%%%%%%%%%%%%%%%%%%%%%%%%%%%%%%%%%%%%%%%%%%%%%%%%%%%%
\section{Crossings, fluctuations and Cauchy convergence}
\label{sec:deterministic}
%%%%%%%%%%%%%%%%%%%%%%%%%%%%%%%%%%%%%%%%%%%%%%%%%%%%%%%%%%%%%%%%%%%%%%%%%%%%%%%%%%%%%%%%%%
%%%%%%%%%%%%%%%%%%%%%%%%%%%%%%%%%%%%%%%%%%%%%%%%%%%%%%%%%%%%%%%%%%%%%%%%%%%%%%%%%%%%%%%%%%

We begin by presenting several characterisations of convergence for sequences of real numbers. The relationship between these notions is well-known, and, with the exception of finite crossings, have been studied from a computational perspective, for example, in the setting of the mean ergodic theorem in \cite{avigad-rute:14:oscillation} and from a deeper logical standpoint in \cite{kohlenbach-safarik:14:fluctuations}. In this section, we complete the picture by making explicit the quantitative relationship between finite crossings and the other notions and, as such, draw a simplified picture (in the nonstochastic setting) of the main ideas we aim to explore for stochastic processes in the remainder of the paper. We also introduce the concept of learnable rates (due to \cite{kohlenbach-safarik:14:fluctuations}, where learnability is defined much more generally) and provide a technical result (Lemma \ref{res:sequence:to:function}) that formulates the existence of simple learnable rates in terms of a simpler property, which will later inspire our definitions of uniform and pointwise stochastic learnability in Section \ref{sec:finitizing}.

In what follows, we write 
\[
[m;n]:=\{m,m+1,\ldots,n-1,n\}
\]
and adopt the convention that $[m;n]=\emptyset$ if $n<m$. We define $\fluc{N}{\varepsilon}{x_n}$ to be the number of $\varepsilon$-fluctuations that occur in the initial segment $\{x_0,\ldots,x_{N-1}\}$ i.e. the maximal $k\in\NN$ such that there exists
\begin{equation*}
i_1<j_1\leq i_2<j_2\leq\ldots \leq i_k<j_k< N \mbox{ with } |x_{i_l}-x_{j_l}|\geq \varepsilon
\end{equation*}
for all $l=1,\ldots,k$. We write 
\begin{equation*}
\flucinf{\varepsilon}{x_n}=\lim_{N\to\infty}\fluc{N}{\varepsilon}{x_n}
\end{equation*}
for the total number of $\varepsilon$-fluctuations that occur in $\seq{x_n}$. Crossings are defined similarly: For any $\alpha<\beta$ we define $\crs{N}{[\alpha,\beta]}{x_n}$ to be the number of times that $\{x_0,\ldots,x_{N-1}\}$ crosses the interval $[\alpha,\beta]$, specifically the maximal $k\in\NN$ such that there exists
\begin{equation*}
i_1<j_1\leq i_2<j_2\leq\ldots \leq i_k<j_k< N \mbox{ with } x_{i_l}\leq \alpha \mbox{ and }\beta\leq x_{j_l} \mbox{ or vice-versa}
\end{equation*}
for all $l=1,\ldots,k$, with
\begin{equation*}
\crsinf{[\alpha,\beta]}{x_n}=\lim_{N\to\infty}\crs{N}{[\alpha,\beta]}{x_n}
\end{equation*}
the total number of $[\alpha,\beta]$-crossings that occur in $\seq{x_n}$. 

\begin{proposition}[Folklore]
The following statements are all equivalent to $\seq{x_n}$ being convergent.
\begin{enumerate}[(a)]

	\item (Cauchy property) For all $\varepsilon>0$ there exists some $n\in\NN$ such that $i,j\geq n$ implies $|x_i-x_j|<\varepsilon$.\smallskip
	
	\item (Finite crossings) $\seq{x_n}$ is bounded and $\crsinf{[\alpha,\beta]}{x_n}<\infty$ for all $\alpha<\beta$.\smallskip
	
	\item (Finite fluctuations) $\flucinf{\varepsilon}{x_n}<\infty$ for all $\varepsilon>0$.\smallskip
	
	\item (Metastability) For all $\varepsilon>0$ and $g:\NN\to\NN$ there exists some $n\in\NN$ such that $|x_i-x_j|<\varepsilon$ for all $i,j\in [n;n+ g(n)]$.

\end{enumerate}
\end{proposition}

It is easy to show that if $\phi(\varepsilon)$ is a rate of Cauchy convergence for $\seq{x_n}$ in the sense that $i,j\geq \phi(\varepsilon)$ implies $|x_i-x_k|<\varepsilon$, then it is also a bound on $\varepsilon$-fluctuations in that $\flucinf{\varepsilon}{x_n}\leq \phi(\varepsilon)$, since any $\varepsilon$-fluctuations have to occur before index $\phi(\varepsilon)$. 

Moreover, if $\phi(\varepsilon)$ is any function satisfying $\flucinf{\varepsilon}{x_n}\leq \phi(\varepsilon)$, then $\Phi(\varepsilon,g):=\tilde{g}^{(\phi(\varepsilon))}(0)$ for $\tilde g(n):=n+g(n)$ is a so-called rate of metastability for $\seq{x_n}$ in the sense that
\begin{equation}
\label{eqn:rateofmetastability:reals}
\exists n\leq \Phi(\varepsilon,g)\, \forall i,j\in [n;n+g(n)](|x_i-x_j|<\varepsilon).
\end{equation}
Supposing for contradiction that this was not the case, then we would have
\begin{equation*}
\exists i,j\in [\tilde{g}^{(e)}(0);\tilde{g}^{(e+1)}(0)](|x_i-x_j|\geq \varepsilon)
\end{equation*}
for all $e=0,\ldots,\phi(\varepsilon)$, and thus $\flucinf{\varepsilon}{x_n}\geq \phi(\varepsilon)+1$.

As such, there is a clear quantitative route from both (a) to (c) and from (c) to (d). Neither of these is (quantitatively) reversible as shown in \cite{kohlenbach-safarik:14:fluctuations}: There exist sequences with a computable bound on $\flucinf{\varepsilon}{x_n}$ but no computable rate of Cauchy convergence, and also sequences with a computable rate of metastability but no computable bound on $\flucinf{\varepsilon}{x_n}$.

%%%%%%%%%%%%%%%%%%%%%%%%%%%%%%%%%%%%%%%%%%%%%%%%%%%%%%%%%%%%%%%%%%%%%%%%%%%%%%%%%%%%%%%%%%
\subsection{Finite crossings}
\label{sec:deterministic:crossings}
%%%%%%%%%%%%%%%%%%%%%%%%%%%%%%%%%%%%%%%%%%%%%%%%%%%%%%%%%%%%%%%%%%%%%%%%%%%%%%%%%%%%%%%%%%

We briefly complete the quantitative picture by showing how (b) fits into the above: Specifically, having a computable bound on finite crossings (along with a bound on the whole sequence) is equivalent to the existence of a computable bound on fluctuations. This is established using a standard idea (see, e.g. the proof of Theorem 27 of \cite{kachurovskii:96:convergence}), several variants of which will be presented later. In probability theory, one typically sees bounds specifically on \emph{upcrossings}, but since upcrossings and downcrossings differ by at most one, a bound on crossings is essentially equivalent. For us, it is slightly easier to reason directly about crossings.

\begin{definition}
\label{def:P:subintervals}
Given some $M>0$ and $l\in\NN$, let $\mathcal{P}(M,l)$ denote the partition of $[-M,M]$ into $l$ equally sized closed subintervals i.e.
\[
\mathcal{P}(M,l):=\left\{\left[-M+\frac{2Mi}{l},-M+\frac{2M(i+1)}{l}\right]\, \Bigl| \, i=0,\ldots,l-1\right\}.
\]
\end{definition}

\begin{proposition}
\label{res:fluctuations:crossings}
Let $\seq{x_n}$ be a sequence of real numbers. 
\begin{enumerate}[(i)]

	\item\label{res:fluctuations:crossings:i} If $\flucinf{\varepsilon}{x_n}\leq \phi(\varepsilon)$ for all $\varepsilon>0$ then $\crsinf{[\alpha,\beta]}{x_n}\leq \phi(\beta-\alpha)$ for all $\alpha<\beta$. \smallskip
	
	\item\label{res:fluctuations:crossings:ii} If $\crsinf{[\alpha,\beta]}{x_n}\leq \psi(\alpha,\beta)$ for all $\alpha<\beta$ and also $|x_n|\leq M$ for all $n\in\NN$ then $\flucinf{\varepsilon}{x_n}\leq \phi(\varepsilon)$ for all $\varepsilon>0$ where
	\begin{equation*}
	\phi(\varepsilon):=l\cdot \max\{\psi(\alpha,\beta)\, | \, [\alpha,\beta]\in \mathcal{P}(M,l)\} \ \ \ \mbox{for } l:=\Bigl\lceil \frac{4M}{\varepsilon}\Bigr\rceil.
	\end{equation*}

\end{enumerate}
\end{proposition}

\begin{proof}
Part (\ref{res:fluctuations:crossings:i}) is immediate, so we focus on proving (\ref{res:fluctuations:crossings:ii}). Fix $\varepsilon>0$ and divide $[-M,M]$ into $l=\lceil 4M/\varepsilon\rceil$ equal subintervals, which we label $I_j=[\alpha_j,\beta_j]$ for $j=1,\ldots,l$. Since $\beta_j-\alpha_j\leq \varepsilon/2$ and $\seq{x_n}$ is contained in $[-M,M]$, a single $\varepsilon$-fluctuation of $\seq{x_n}$ crosses at least one of the $I_j$. Therefore if $\flucinf{\varepsilon}{x_n}\geq k$ then there must be some interval $[\alpha_j,\beta_j]$ with at least $k/l$ crossings i.e.
\begin{equation*}
\crsinf{[\alpha_j,\beta_j]}{x_n}\geq \frac{k}{l}
\end{equation*}
for this particular $j$, and thus
\begin{equation*}
k\leq l\cdot\psi(\alpha_j,\beta_j)
\end{equation*}
from which the bound follows.
\end{proof}

%%%%%%%%%%%%%%%%%%%%%%%%%%%%%%%%%%%%%%%%%%%%%%%%%%%%%%%%%%%%%%%%%%%%%%%%%%%%%%%%%%%%%%%%%%
\subsection{Monotone sequences}
\label{sec:deterministic:monotone}
%%%%%%%%%%%%%%%%%%%%%%%%%%%%%%%%%%%%%%%%%%%%%%%%%%%%%%%%%%%%%%%%%%%%%%%%%%%%%%%%%%%%%%%%%%

In the very simple case of monotone bounded sequences, we can put the previous sections together to obtain what is essentially the canonical rate of metastability in this case as in Theorem \ref{res:finitary:monotone:convergence}, though the detour through crossings adds a factor of $2$. Specifically, suppose that $\seq{x_n}\subseteq [-M,M]$ is monotone. Then, in particular, the sequence has at most one $[\alpha,\beta]$-crossing, i.e. $\crsinf{[\alpha,\beta]}{x_n}\leq 1$ for all $\alpha<\beta$. By Proposition \ref{res:fluctuations:crossings} it therefore follows that
\begin{equation*}
\flucinf{\varepsilon}{x_n}\leq \Bigl\lceil \frac{4M}{\varepsilon}\Bigr\rceil
\end{equation*}
for any $\varepsilon>0$, and thus a rate of metastability for $\seq{x_n}$ is given by
\begin{equation*}
\Phi(\varepsilon,g)=\tilde g^{(\lceil 4M/\varepsilon\rceil)}(0).
\end{equation*}

%%%%%%%%%%%%%%%%%%%%%%%%%%%%%%%%%%%%%%%%%%%%%%%%%%%%%%%%%%%%%%%%%%%%%%%%%%%%%%%%%%%%%%%%%%
\subsection{Learnable rates}
\label{sec:deterministic:learnability}
%%%%%%%%%%%%%%%%%%%%%%%%%%%%%%%%%%%%%%%%%%%%%%%%%%%%%%%%%%%%%%%%%%%%%%%%%%%%%%%%%%%%%%%%%%

The rate of metastability for monotone bounded sequences has a particularly elegant form, namely an iteration $\tilde g^{(e)}(0)$ of the modified function $\tilde g$. We call such rates \emph{learnable}, loosely following the terminology of \cite{kohlenbach-safarik:14:fluctuations} where the Cauchy property forms a simple instance of the class of \emph{effectively learnable} formulas. The semantic idea behind this terminology is that to find an interval $[n;n+g(n)]$ of local stability, we first just guess $n:=0$, and if this does not work, we update our guess to $n:=\tilde g(0)$, and then $n:=\tilde g(\tilde g(0))$, and so on, knowing that there is an upper bound on how many `mind-changes' are required. Learnable rates (more generally construed) arise routinely from the analysis of convergence proofs \cite{kohlenbach-safarik:14:fluctuations} and also more generally from nonconstructive proofs \cite{aschieri:11:phd,powell:16:learning}.

As we will see below, this phenomenon extends to the stochastic setting, and in our case, it is both convenient and insightful to identify mathematical properties that correspond directly to the existence of learnable metastable rates. The following lemma is particularly helpful in this regard:

\begin{lemma}
    \label{res:sequence:to:function}
    Let $a_0<b_0\leq a_1<b_1\leq \ldots$ be sequences of natural numbers. Then we can define a function $g:\NN\to\NN$ in such a way that $\tilde g^{(i)}(0)=b_{i-1}$ for $i\geq 1$, and more generally such that for any $n\in\NN$ we have $[a_m;b_m]\subseteq [n;n+g(n)]$ for the least $m$ such that $n\leq a_m$.
\end{lemma}

\begin{proof}
    Define $g(n):=b_{k(n)}-n$ where
    \[
    k(n):=\min\{i \, \mid \, n\leq a_i\}
    \]
    (this is well-defined, since $a_0<a_1<\ldots$ and $n\leq a_{k(n)}<b_{k(n)}$). By an easy induction we can show that $\tilde g^{(i)}(0)=b_{i-1}$, where in particular we have $g(b_{i-1})=b_{k(b_{i-1})}-b_{i-1}=b_i-b_{i-1}$. Now for any $n\in\NN$ we have $n+g(n)=b_{k(n)}$ and since $n\leq a_{k(n)}$ it follows that $[a_{k(n)};b_{k(n)}]\subseteq [n;n+g(n)]$.
\end{proof}

As a result, we see that a sequence $\seq{x_n}$ having a learnable rate of metastability, in the sense that for some function $\phi(\varepsilon)$:
\begin{equation}
\label{eqn:metastable:normal}
\exists n\leq \tilde g^{(\phi(\varepsilon))}(0)\, \forall i,j\in [n;n+g(n)](|x_i-x_j|<\varepsilon)
\end{equation}
 is equivalent to the more natural property that for any $a_0<b_0\leq a_1<b_1\leq \ldots$
\begin{equation}
\label{eqn:metastable:learnable}
\exists n\leq \phi(\varepsilon)\, \forall i,j\in [a_n;b_n](|x_i-x_j|<\varepsilon)
\end{equation}
To see that (\ref{eqn:metastable:learnable}) implies (\ref{eqn:metastable:normal}) we just define $a_n:=\tilde g^{(n)}(0)$ and $b_n:=\tilde g^{(n+1)}(0)$ for all $n\leq \phi(\varepsilon)$, and some arbitrary increasing sequence from that point. Then if (\ref{eqn:metastable:normal}) is false, we have $a_n<b_n$ for all $n\leq \phi(e)$ and (\ref{eqn:metastable:learnable}) must also be false. Conversely, from $a_0<b_0\leq a_1<b_1\leq \ldots$ we define $g$ as in Lemma \ref{res:sequence:to:function}, and then if (\ref{eqn:metastable:learnable}) is false then by Lemma \ref{res:sequence:to:function}, for any $n\leq  \tilde g^{(\phi(\varepsilon))}(0)=b_{\phi(\varepsilon)-1}\leq a_{\phi(\varepsilon)}$ we have $[a_m;b_m]\subseteq [n;n+g(n)]$ for some $m\leq \phi(\varepsilon)$, and thus (\ref{eqn:metastable:normal}) is also false. Observe that $\phi(\varepsilon)$ satisfying (\ref{eqn:metastable:learnable}) is also a bound for $\flucinf{\varepsilon}{x_n}$. Therefore, the above and the discussion just before Section \ref{sec:deterministic:crossings} demonstrates the equivalence between a bound on the $\varepsilon$-fluctations and the exponent of a learnable rate of metastability.

Similar equivalences will be presented in the stochastic setting, where they will inspire several of our main quantitative notions of stochastic convergence.

%%%%%%%%%%%%%%%%%%%%%%%%%%%%%%%%%%%%%%%%%%%%%%%%%%%%%%%%%%%%%%%%%%%%%%%%%%%%%%%%%%%%%%%%%%
%%%%%%%%%%%%%%%%%%%%%%%%%%%%%%%%%%%%%%%%%%%%%%%%%%%%%%%%%%%%%%%%%%%%%%%%%%%%%%%%%%%%%%%%%%
\section{Quantitative almost sure statements}
\label{sec:finitizing}
%%%%%%%%%%%%%%%%%%%%%%%%%%%%%%%%%%%%%%%%%%%%%%%%%%%%%%%%%%%%%%%%%%%%%%%%%%%%%%%%%%%%%%%%%%
%%%%%%%%%%%%%%%%%%%%%%%%%%%%%%%%%%%%%%%%%%%%%%%%%%%%%%%%%%%%%%%%%%%%%%%%%%%%%%%%%%%%%%%%%%

Before presenting our main results, we outline a general approach to providing quantitative analogues to probabilistic statements, including the property that a stochastic process converges almost surely. Generally, this involves pulling quantifiers out of the probability function and sometimes shuffling the quantifiers to arrive at `finitisations' of the original properties. This process of finitisation plays an important role in various fields of analysis: It is discussed as a general concept by Tao in \cite{tao:07:softanalysis} where several examples of applications are given, and finitary, or \emph{metastable} almost sure convergence statements are already explored from a logical perspective in \cite{avigad-dean-rute:12:dominated}.  

Let us fix some probability space $(\Omega,\mathcal{F},\PP)$. We say that a logical formula $\varphi(x_1,\ldots,x_n)$ with parameters $x_1,\ldots,x_n$ is \emph{measurable} when for any parameters we can regard the formula as a measurable event i.e. $\varphi(x_1,\ldots,x_n)\in \mathcal{F}$. If $\varphi(n)$ is measurable for all $n\in\NN$, we define $\exists n\, \varphi(n),\forall n\, \varphi(n)\in \mathcal{F}$ in the expected way:
\begin{equation*}
\exists n\, \varphi(n):=\bigcup_{n\in\NN}\varphi(n) \ \ \ \mbox{and} \ \ \ \forall n\, \varphi(n):=\bigcap_{n\in\NN}\varphi(n).
\end{equation*}
We write $A^c$ for the complement of an event $A$. The following straightforward facts will be used repeatedly:
\begin{lemma}
\label{res:bring:out:quantifiers}
Let $p\in [0,1]$ and suppose that $\varphi(n)\supseteq\varphi(n+1)$ for all $n\in\NN$. Then
\begin{enumerate}[(i)]

	\item $\PP(\forall n\, \varphi(n))\geq p\iff \forall n\, (\PP(\varphi(n))\geq p)$.\smallskip
	
	\item $\PP(\forall n\, \varphi(n))\leq p\iff \forall \lambda>0\, \exists n\, (\PP(\varphi(n))<p+\lambda)$.

\end{enumerate}
On the other hand, if $\varphi(n)\subseteq\varphi(n+1)$ for all $n\in\NN$ then

\begin{enumerate}[(i)]

	\item[(iii)] $\PP(\exists n\, \varphi(n))\leq p\iff \forall n\, (\PP(\varphi(n))\leq p)$.\smallskip
	
	\item[(iv)] $\PP(\exists n\, \varphi(n))\geq p\iff \forall \lambda>0\, \exists n\, (\PP(\varphi(n))> p-\lambda)$.

\end{enumerate}\end{lemma}

\begin{proof}
Parts (i) and (ii) follow directly from the fact that $\seq{\PP(\varphi(n))}$ is a decreasing sequence of reals with
\begin{equation*}
\lim_{n\to\infty}\PP(\varphi(n))=\PP(\forall n\, \varphi(n)).
\end{equation*}
If $\PP(\forall n\, \varphi(n))\geq p$ then $\PP(\varphi(n))\geq \PP(\forall n\, \varphi(n))\geq p$ for any $n\in\NN$, and conversely if $\PP(\varphi(n))\geq p$ for all $n\in\NN$, we must have $\PP(\forall n\, \varphi(n))=\lim_{n\to\infty}\PP(\varphi(n))\geq p$. Similarly, for (ii), if $\PP(\forall n\, \varphi(n))=\lim_{n\to\infty}\PP(\varphi(n))\leq p$ then in particular, for any $\lambda>0$ we have $\PP(\varphi(n))<p+\lambda$ for some $n\in\NN$, and conversely if for any $\lambda$ we have $\PP(\varphi(n))<p+\lambda$ for some $n\in\NN$, since $\seq{\PP(\varphi(n))}$ is decreasing we have $\PP(\forall n\, \varphi(n))=\lim_{n\to\infty}\PP(\varphi(n))< p+\lambda$ for all $\lambda>0$, and thus $\PP(\forall n\, \varphi(n))\leq p$. Parts (iii) and (iv) follow by negating both sides of the implications and applying (i) and (ii) to the complement of $\varphi(n)$.
\end{proof}

%%%%%%%%%%%%%%%%%%%%%%%%%%%%%%%%%%%%%%%%%%%%%%%%%%%%%%%%%%%%%%%%%%%%%%%%%%%%%%%%%%%%%%%%%%
\subsection{Pointwise and uniform metastability}
\label{sec:finitizing:metastability}
%%%%%%%%%%%%%%%%%%%%%%%%%%%%%%%%%%%%%%%%%%%%%%%%%%%%%%%%%%%%%%%%%%%%%%%%%%%%%%%%%%%%%%%%%%

We now introduce an abstract notion of what it means to `finitise' an almost sure statement, which will later be applied concretely to convergence statements (where such finitisations already feature in \cite{avigad-dean-rute:12:dominated}) but which is helpful to have in a more general form. In the remainder of this subsection, we let $A(n,m)$ be some measurable formula on natural numbers $n,m\in\NN$, with the following monotonicity property: $n\leq n'$ and $m'\leq m$ implies that $A(n',m')\subseteq A(n,m)$.

\begin{theorem}
    \label{res:equivalent:metastabe}
    The following statements are equivalent:
    \begin{enumerate}[(a)]
        \item Almost surely, there exists $n$ such that $A(n,m)$ does not hold for any $m\geq n$. \smallskip

        \item For any $\lambda>0$ and $g:\NN\to\NN$ there exists $n$ such that 
        \[
        \PP(A(n,n+g(n)))<\lambda
        \]

        \item For any $\lambda>0$ and $g:\NN\to\NN$ there exists $N$ such that
        \[
        \PP(\forall n\leq N\, A(n,n+g(n)))<\lambda
        \]     
    \end{enumerate}
\end{theorem}

\begin{proof}
    The equivalence of (a) and (b) follows through repeated applications of Lemma \ref{res:bring:out:quantifiers}. Specifically, (a) can be reformulated as
    \[
    \PP(\exists n\, \forall m\, A(n,n+m)^c)=1
    \]
    which, using the monotonicity property of $A$ and (iv) applied to $\varphi(n):=\forall m\, A(n,n+m)^c$ (and replacing $>1-\lambda$ with $\geq 1-\lambda$), is equivalent to
    \[
    \forall \lambda>0\, \exists n\, \PP(\forall m\, A(n,n+m)^c)\geq 1-\lambda
    \]
    Now using (i) applied to $\varphi(m):=A(n,n+m)^c$, this is equivalent to
    \[
     \forall \lambda>0\, \exists n\, \forall m\, \PP(A(n,n+m)^c)\geq 1-\lambda   
    \]
		which is equivalent to
		\[
		     \forall \lambda>0\, \exists n\, \forall m\, \PP(A(n,n+m))<\lambda   
		\]
    Now, if the above is false, that is equivalent to the statement that there exists some $\lambda>0$ and $g:\NN\to\NN$ such that
    \[
    \forall n\, \PP(A(n,n+g(n)))\geq \lambda
    \]
    which is then the negation of (b). That (b) implies (c) is clear. The converse is more difficult, and we use the abstract quantitative version of Egorov's theorem formulated as Theorem 2.2 of \cite{avigad-dean-rute:12:dominated}, which is in turn inspired by a construction of Tao \cite[Theorem A.2]{tao:08:ergodic}. Theorem 2.2 of \cite{avigad-dean-rute:12:dominated} implies the following simpler statement: for any sequence of events $\seq{B_n}$ and any $\lambda>\lambda'>0$, if for any $F:\NN\to\NN$ there exists an $N$ such that
    \[
    \PP\left(\forall n\leq N\, \exists k\in [n;F(n)]\, B_k\right)<\lambda'
    \]
    then $\PP(B_n)<\lambda$ for some $n$. To prove (b) from this statement, fixing $\lambda$ and $g$ and defining
    \[
    B^g_n:=A(n,n+g(n))
    \]
    it suffices to show that for some $\lambda'<\lambda$, for all $F$ there exists $N$ such that
    \[
    \PP\left(\forall n\leq N\, \exists k\in [n;F(n)]\, B_k^g\right)<\lambda'.
    \]
    But using monotonicity of $A$ by which we have
		\[
		\begin{aligned}
		\exists k\in [n;F(n)]\, B_k^g&=\exists k\in [n;F(n)]\, A(k,k+g(k))\\
		&\subseteq \exists k\in [n;F(n)]\, A(n,k+g(k))\\
		&\subseteq A(n,n+F^g(n))
		\end{aligned}
		\]
		for $F^g(n):=\max\{k-n+g(k)\, \mid \, k\in [n;F(n)]\}$, it suffices to show that
     \[
    \PP\left(\forall n\leq N\, A(n,n+F^g(n))\right)<\lambda'
    \]   
    for any $F$, and the existence of such an $N$ then follows from (c).
\end{proof}

We now arrive at the following definitions, each of which gives a general quantitative meaning to the statement $\exists n\, \forall m\geq n\, A(n,m)^c$ almost surely.
\begin{definition}
    \label{def:abstract:rates}
    Let $B:=\exists n\, \forall m\geq n\, A(n,m)^c$. Then
    \begin{enumerate}[(a)]

        \item A (direct) rate for $B$ is any function $f:(0,1)\to\NN$ satisfying
        \[
        \PP(\exists m\geq f(\lambda)\, A(f(\lambda),m))<\lambda
        \]
        for all $\lambda>0$.\smallskip

        \item A uniform metastable rate for $B$ is any functional $\Phi:(0,1)\times (\NN\to\NN)\to \NN$ satisfying
        \[
        \exists n\leq \Phi(\lambda,g)\, \PP(A(n,n+g(n)))<\lambda
        \]
        for all $\lambda>0$ and $g:\NN\to\NN$.\smallskip

        \item A pointwise metastable rate for $B$ is any function $\Phi:(0,1)\times (\NN\to\NN)\to \NN$ satisfying
        \[
        \PP(\forall n\leq \Phi(\lambda,g)\, A(n,n+g(n)))<\lambda
        \]
        for all $\lambda>0$ and $g:\NN\to\NN$.        
        
    \end{enumerate}
\end{definition}
In general, and certainly for the specific case of almost sure convergence (cf. Remark \ref{rem:specker}), it is not possible to find \emph{direct} rates that are computable. On the other hand, computable pointwise and uniform rates are generally possible. Any uniform rate is automatically a pointwise rate, while the construction given in \cite{avigad-dean-rute:12:dominated} gives us a general method of converting pointwise rates to uniform rates, though with a considerable blowup in complexity (cf. Remark \ref{rem:blowup}).  So, even in situations where we do have pointwise rates, it is not necessarily obvious how to find uniform rates of low complexity.

%%%%%%%%%%%%%%%%%%%%%%%%%%%%%%%%%%%%%%%%%%%%%%%%%%%%%%%%%%%%%%%%%%%%%%%%%%%%%%%%%%%%%%%%%%
\subsection{Pointwise and uniform learnability}
\label{sec:finitizing:learnability}
%%%%%%%%%%%%%%%%%%%%%%%%%%%%%%%%%%%%%%%%%%%%%%%%%%%%%%%%%%%%%%%%%%%%%%%%%%%%%%%%%%%%%%%%%%

We now consider the special case where rates of pointwise and uniform metastability have learnable forms, as already discussed in Section \ref{sec:deterministic:learnability}. In both cases, we demonstrate that the existence of learnable rates can be formulated in a simpler way (without reference to the function $g$) and that there are very natural probabilistic situations where these conditions arise. In the following result, we assume that $A(m,n)$ is monotone as in Theorem \ref{res:equivalent:metastabe} above, and also that $A(m,n)=\emptyset$ unless $m<n$.

\begin{definition}
    \label{def:learnability}
    Let $B:=\exists n\, \forall m\geq n\, A(n,m)^c$. A function $\phi:(0,1)\to\NN$ is 
    \begin{enumerate}[(a)]

        \item a uniform learnable rate for $B$ if 
        \[
        \exists n\leq \phi(\lambda)\, \PP(A(a_n,b_n))<\lambda
        \]
        for any $a_0<b_0\leq a_1<b_1\leq \ldots$

        \item a pointwise learnable rate for $B$ if
        \[
        \PP(\forall n\leq \phi(\lambda)\, A(a_n,b_n))<\lambda
        \]
        for any $a_0<b_0\leq a_1<b_1\leq \ldots$

    \end{enumerate}
\end{definition}

\begin{lemma}
    \label{res:learnable:equivalences} Let $B:=\exists n\, \forall m\geq n\, A(n,m)^c$ and $\phi:(0,1)\to\NN$ be some function. Then $\Phi(\lambda,g)=\tilde g^{(\phi(\lambda))}(0)$ for $\tilde g(n):=n+g(n)$ is a (uniform) pointwise metastable rate for $B$ iff $\phi(\lambda)$ a (uniform) pointwise learnable rate for $B$.
\end{lemma}

\begin{proof}
    For the uniform case, in one direction we define $a_n:=\tilde g^{(n)}(0)$ and $b_n:=\tilde g^{(n+1)}(0)$. Then if $\tilde g^{(\phi(\lambda))}(0)$ is not a uniform metastable rate then
    \[
    \forall n\leq \phi(\lambda)\, \PP(A(a_n,b_n))\geq \lambda
    \]
    and since then we would have $a_n<b_n$ for all $n\leq \phi(\lambda)$ then $\phi(\lambda)$ is not a uniform learnable rate. In the other direction, we define $g$ in terms of $a_0<b_0\leq a_1<b_1\leq \ldots$ as in Lemma \ref{res:sequence:to:function}, and if $\phi(\lambda)$ is not a uniform learnable rate then since for any $n\leq \tilde g^{(\phi(\lambda))}(0)=b_{\phi(\lambda)-1}\leq a_{\phi(\lambda)}$ we have $n\leq a_m$ and $b_m=n+g(n)$ for some $m\leq \phi(\lambda)$, and since $A(a_m,b_m)\subseteq A(n,n+g(n))$ it follows that $\PP(A(n,n+g(n)))\geq \lambda$ for all $n\leq \tilde g^{(\phi(\lambda))}(0)$.

    The pointwise case is entirely analogous, with a few additional details needed to run the argument pointwise. For the first direction, defining $a_n,b_n$ in the same way, we note that if $\omega\in  A(n,n+g(n))$ for all $n\leq \tilde g^{(\phi(\lambda))}(0)$, then in particular $\omega\in A(a_n,b_n)$ for $n\leq \phi(\lambda)$, and so if $g^{(\phi(\lambda))}(0)$ is not a pointwise metastable rate then
		\[
		\PP(\forall n\leq \phi(\lambda)\, A(a_n,b_n))\geq \PP(\forall n\leq \tilde g^{(\phi(\lambda))}(0)\, A(n,n+g(n)))\geq \lambda
		\]
		where we must also note that for any $n\leq \phi(\lambda)$,
		\[
		\PP(A(a_n,b_n))\geq \PP(\forall n\leq \phi(\lambda)\, A(a_n,b_n))\geq \lambda
		\]
		and thus $a_n<b_n$. In the other direction, we note that that if $\omega\in A(a_n,b_n)$ for all $n\leq \phi(\lambda)$, defining $g$ in the same way as the uniform case, for any $n\leq \tilde g^{(\phi(\lambda))}(0)$ there exists $m\leq \phi(\lambda)$ such that $\omega\in A(a_m,b_m)\subseteq A(n,n+g(n))$, and so $\omega\in A(n,n+g(n))$ for all $n\leq \tilde g^{(\phi(\lambda))}(0)$, and thus if $\phi(\lambda)$ is not a pointwise learnable rate then
    \[
    \PP(\forall n\leq \tilde g^{(\phi(\lambda))}(0)\, A(n,n+g(n)))\geq \PP(\forall n\leq \phi(\lambda)\, A(a_n,b_n))\geq \lambda
    \]
    from which we obtain our contradiction.
\end{proof}

The concepts of pointwise and uniform learnability above, which in turn correspond to rates of uniform and pointwise metastability with a particularly clean form, are, in turn, fundamentally related to an abstract notion of `fluctuations'. Define the random variable $J_{N,A}$ to be the maximal $k\in\NN$ such that there exists
\[
a_0<b_0\leq a_1<b_1\leq \ldots \leq a_{k-1}<b_{k-1}< N \mbox{ with }A(a_n,b_n)\mbox{ for all }n\in [0;k-1]
\]
and define
\[
J_{\infty,A}:=\lim_{N\to\infty} J_{N,A}
\]

\begin{theorem}
    \label{res:abstract:fluctuations:to:rates}
    Let $B=\exists n\, \forall m\geq n\, A(n,m)^c$.
    \begin{enumerate}[(i)]

        \item If $\phi:(0,1)\to\NN$ is a rate of convergence for
        \[
        \lim_{n\to\infty}\PP\left(J_{\infty,A}\geq n\right)=0
        \]
        then it is also a learnable pointwise rate for $B$, and thus $\Phi(\lambda,g):=\tilde g^{(\phi(\lambda))}(0)$ is a pointwise metastable rate for $B$.\smallskip
        \item If $K>0$ is such that
        \[
        \EE[J_{\infty,A}]< K
        \]
        then $\phi(\lambda):= \lceil K/\lambda\rceil$ is a uniform learnable rate for $B$, and thus $\Phi(\lambda,g):={\tilde g^{(\lceil K/\lambda\rceil)}(0)}$ is a uniform metastable rate for $B$. 
        
    \end{enumerate}
\end{theorem}

\begin{proof}
    Fixing $a_0<b_1\leq a_1<b_1\leq \ldots$, part (i) follows by noting that 
    \[
    \{\forall n\leq \phi(\lambda)\, A(a_n,b_n)\}\subseteq \{J_{\infty,A}\geq \phi(\lambda)\}
    \]
    while for (ii), we observe that for any $N\in\NN$:
    \[
    \sum_{i=0}^{N}\PP(A(a_n,b_n))=\sum_{i=0}^{N}\EE[I_{A(a_n,b_n)}]=\EE\left(\sum_{i=0}^{N}I_{A(a_n,b_n)}\right)\leq \EE[J_{\infty,A}]<K
    \]
    and therefore if $\PP(A(a_n,b_n))\geq \lambda$ for all $n\leq \lceil K/\lambda\rceil$, we would have
    \[
    K<\lambda\left(\Bigl\lceil \frac{K}{\lambda}\Bigr\rceil+1\right)\leq \sum_{i=0}^{\lceil K/\lambda\rceil}\PP(A(a_n,b_n))< K
    \]
    a contradiction.
\end{proof}

Theorem \ref{res:abstract:fluctuations:to:rates} formulates, in a completely abstract way, an instance of a deeper correspondence which, as we will see, plays an important role in the quantitative analysis of stochastic processes that follow, namely:
\[
\begin{aligned}
    \lim_{n\to\infty}\PP\left(Z\geq n\right)=0&\mapsto \mbox{pointwise local stability}\\
    \EE[Z]<\infty&\mapsto \mbox{uniform local stability}
\end{aligned}
\]
where $Z$ is some random variable representing counterexamples to stability in a suitable sense. In addition to fluctuations, we will also see this relationship with the crossings of stochastic processes over fixed intervals, though here, the route to pointwise and uniform convergence rates is a little more involved.

%%%%%%%%%%%%%%%%%%%%%%%%%%%%%%%%%%%%%%%%%%%%%%%%%%%%%%%%%%%%%%%%%%%%%%%%%%%%%%%%%%%%%%%%%%
%%%%%%%%%%%%%%%%%%%%%%%%%%%%%%%%%%%%%%%%%%%%%%%%%%%%%%%%%%%%%%%%%%%%%%%%%%%%%%%%%%%%%%%%%%
\section{Proof-theoretic rates and moduli}
\label{sec:moduli}
%%%%%%%%%%%%%%%%%%%%%%%%%%%%%%%%%%%%%%%%%%%%%%%%%%%%%%%%%%%%%%%%%%%%%%%%%%%%%%%%%%%%%%%%%%
%%%%%%%%%%%%%%%%%%%%%%%%%%%%%%%%%%%%%%%%%%%%%%%%%%%%%%%%%%%%%%%%%%%%%%%%%%%%%%%%%%%%%%%%%%

In this section, we apply the ideas and manipulations of Section \ref{sec:finitizing} to arrive at a number of quantitative definitions concerning stochastic variants of the main concepts discussed in Section \ref{sec:deterministic}, specifically a stochastic process: 
\begin{itemize}
\item being uniformly bounded;
\item having finite fluctuations or crossings;
\item being Cauchy convergent.
\end{itemize}
Our definitions are based on functions that reflect the logical structure of the underlying statements, which is often at odds with the more traditional formulations. For example, we capture the uniform boundedness of a stochastic process with a function $\phi$ satisfying
\[
\PP\left(\sup_{n\in\NN}|X_n|\geq\phi(\lambda)\right)<\lambda \ \ \ \mbox{for all $\lambda>0$}
\]
rather than a function $f$ satisfying
\[
\PP\left(\sup_{n\in\NN}|X_n|\geq m\right)<f(m) \ \ \ \mbox{for all $m\in\NN$}.
\]
It is precisely because they represent the underlying quantifier structure that these moduli are better suited to formulating the computational structure of proofs than traditional rates, which implicitly involve additional assumptions, such as \[\lim_{m\to\infty}f(m)= 0\] in the example above. In any case, we can always convert our moduli to traditional rates to facilitate comparison with known results, which we do in Section \ref{sec:pointwise} below.

%%%%%%%%%%%%%%%%%%%%%%%%%%%%%%%%%%%%%%%%%%%%%%%%%%%%%%%%%%%%%%%%%%%%%%%%%%%%%%%%%%%%%%%%%%
\subsection{Uniform boundedness}
\label{sec:moduli:boundedness}
%%%%%%%%%%%%%%%%%%%%%%%%%%%%%%%%%%%%%%%%%%%%%%%%%%%%%%%%%%%%%%%%%%%%%%%%%%%%%%%%%%%%%%%%%%

Let $\seq{X_n}$ be a stochastic process. By Lemma \ref{res:bring:out:quantifiers} (iv), the property that $\seq{X_n}$ is almost surely uniformly bounded i.e.
\begin{equation*}
\sup_{n\in\NN} |X_n|<\infty \ \ \ \mbox{almost surely}
\end{equation*}
is equivalent to the statement that for any $\lambda>0$ there exists $N\in\NN$ such that
\begin{equation*}
\PP\left(\sup_{n\in\NN} |X_n| \geq N\right)<\lambda.
\end{equation*}
Uniform boundedness is related to the notion of \emph{tightness}. In particular, it implies that the sequence $\seq{X_n}$ is tight in the sense that for any $\lambda>0$ there exists $N\in\NN$ such that
\begin{equation*}
\PP\left(|X_n| \geq N\right)<\lambda \ \ \ \mbox{for all $n\in\NN$}.
\end{equation*}
Tightness is strictly weaker than almost sure uniform boundedness. In particular, whenever
\begin{equation*}
\sup_{n\in\NN}\EE(|X_n|)<\infty
\end{equation*}
then $\seq{X_n}$ is tight by Markov's inequality but is not necessarily almost surely bounded (see Example \ref{ex:tightness} below). In line with the discussion at the end of Section \ref{sec:finitizing}, uniform boundedness will be needed to obtain results related to pointwise convergence, but this can be relaxed to tightness in the case of uniform convergence. We therefore consider two distinct quantitative notions.
\begin{definition}
\label{def:quantitative:boundedness}
Let $\seq{X_n}$ be a stochastic process.\medskip
\begin{enumerate}[(a)]

	\item Any function $\phi:(0,1)\to \RR$ satisfying
	\begin{equation*}
	\PP\left(\sup_{n\in\NN}|X_n|\geq \phi(\lambda)\right)<\lambda \ \ \ \mbox{for all $\lambda\in (0,1)$}
	\end{equation*}
	is called a \emph{modulus of uniform boundedness} for $\seq{X_n}$.\medskip
	
	\item Any function $\phi:(0,1)\to \RR$ satisfying
	\begin{equation*}
	\PP\left(|X_n|\geq \phi(\lambda)\right)<\lambda \ \ \ \mbox{for all $\lambda\in (0,1)$ and $n\in\NN$}
	\end{equation*}
	is called a \emph{modulus of tightness} for $\seq{X_n}$.\medskip
	
\end{enumerate}
In particular, any modulus of uniform boundedness is also a modulus of tightness for the same stochastic process.
\end{definition} 

By a simple application of Markov's inequality, we obtain the following:

\begin{lemma}
\label{res:markov:tightness}
Suppose that 
\begin{equation*}
\sup_{n\in\NN}\EE(|X_n|)< M
\end{equation*}
for some $M>0$. Then $\seq{X_n}$ is tight with modulus $\phi_M(\lambda):=M/\lambda$.
\end{lemma}

We finish with some very simple examples of concrete moduli:

\begin{example}
\label{ex:tightness}
Define $\seq{X_n}$ by
\begin{equation*}
\begin{aligned}
&X_0=1\\
&X_1=2I_{[0,1/2]}, X_2=2I_{[1/2,1]}\\
&X_3=3I_{[0,1/3]}, X_4=3I_{[1/3,2/3]}, X_5=3I_{[2/3,1]}\\
&\cdots
\end{aligned}
\end{equation*}
Then $\EE(|X_n|)=\EE(X_n)=1$ for all $n\in\NN$, and thus $\seq{X_n}$ is tight with modulus $\phi(\lambda)=1/\lambda$. On the other hand, for any $N\in\NN$ we have
\begin{equation*}
\PP\left(\sup_{n\in\NN }X_n\geq N\right)=1
\end{equation*}
and so $\seq{X_n}$ is almost surely \emph{unbounded}.
\end{example}

\begin{example}
\label{ex:maximal}
If $\seq{X_n}$ is a nonnegative submartingale with $\sup_{n\in\NN}\EE(|X_n|)<M$, then it is both tight and almost surely uniformly bounded, with a modulus $\phi(\lambda)=M/\lambda$ in both cases. The latter follows from Doob's maximal inequality, whereby
\begin{equation*}
\PP\left(\max_{n\leq m} X_n\geq N\right)\leq\frac{\EE(X_m)}{N}<\frac{M}{N}
\end{equation*}
for any $m\in\NN$, and thus by Lemma \ref{res:bring:out:quantifiers} (iii) it follows that
\begin{equation*}
\PP\left(\sup_{n\in\NN}X_n\geq N\right)=\PP\left(\exists m\, \max_{n\leq m} X_n\geq N\right)<\frac{M}{N}
\end{equation*}

\end{example}

%%%%%%%%%%%%%%%%%%%%%%%%%%%%%%%%%%%%%%%%%%%%%%%%%%%%%%%%%%%%%%%%%%%%%%%%%%%%%%%%%%%%%%%%%%
\subsection{Fluctuations and crossings}
\label{sec:moduli:crossings}
%%%%%%%%%%%%%%%%%%%%%%%%%%%%%%%%%%%%%%%%%%%%%%%%%%%%%%%%%%%%%%%%%%%%%%%%%%%%%%%%%%%%%%%%%%

For a stochastic process $\seq{X_n}$ define 
\begin{equation*}
\fluc{N}{\varepsilon}{X_n}(\omega):=\fluc{N}{\varepsilon}{X_n(\omega)}
\end{equation*}
for each $\omega \in \Omega$, with $\fluc{N}{\varepsilon}{x_n}$ defined as in Section \ref{sec:deterministic}. In other words, $\fluc{N}{\varepsilon}{X_n}$ denotes the number of $\varepsilon$-fluctuations that occur in the initial segment $\{X_0,\ldots,X_{N-1}\}$ of the process. The stochastic analogue of total fluctuations $\flucinf{\varepsilon}{X_n}$, along with those for $[\alpha,\beta]$-crossings $\crs{N}{[\alpha,\beta]}{X_n}$ and $\crsinf{[\alpha,\beta]}{X_n}$, are defined in the same way.

The property that, almost surely, $\seq{X_n}$ has finite $\varepsilon$-fluctuations for each $\varepsilon>0$, is equivalent to
\begin{equation*}
\PP\left(\forall k\, \flucinf{2^{-k}}{X_n}<\infty\right)=1
\end{equation*}
and Lemma \ref{res:bring:out:quantifiers} and monotonicity of $\flucinf{2^{-k}}{X_n}<\infty$ in $k$ this is equivalent to
\begin{equation*}
\forall k\, \left(\PP\left(\flucinf{2^{-k}}{X_n}<\infty\right)=1\right)
\end{equation*}
i.e. for any $\varepsilon>0$, $\seq{X_n}$ has finite $\varepsilon$-fluctuations almost surely. This leads to the following quantitative notion:
\begin{definition}
\label{def:quantitative:fluctuations}
Let $\seq{X_n}$ be a stochastic process. For fixed $\varepsilon>0$, any function $\phi:(0,1)\to \RR$ satisfying
	\begin{equation*}
	\PP\left(\flucinf{\varepsilon}{X_n}\geq \phi(\lambda)\right)<\lambda \ \ \ \mbox{for all $\lambda\in (0,1)$}
	\end{equation*}
	is called a \emph{modulus of finite $\varepsilon$-fluctuations} for $\seq{X_n}$. Any function $\phi:(0,1)\times (0,1)\to \RR$ such that $\phi(\cdot,\varepsilon)$ is a modulus of finite $\varepsilon$-fluctuations for all $\varepsilon\in (0,1)$ is simply called a \emph{modulus of finite fluctuations} for $\seq{X_n}$.
\end{definition}

A modulus of finite fluctuations is just another way of formulating the rate of convergence of
\begin{equation*}
\lim_{N\to\infty}\PP(\flucinf{\varepsilon}{X_n}\geq N)=0
\end{equation*}
a quantitative notion that has been widely explored, particularly in the context of martingales (e.g. \cite{kachurovskii:96:convergence}). Several results in this direction will be quoted later.

In order to bring quantifiers out of the statement that, almost surely, a stochastic process has finite $[\alpha,\beta]$-crossings for all intervals $[\alpha,\beta]$, we must find a way to encode quantification over all intervals in a monotone way. Our choice of encoding, which in turn informs our definition of the corresponding modulus, reflects the way that crossings are used in the convergence proofs we analyse. To be precise, having finite crossings over arbitrary intervals almost surely is equivalent to the statement
\begin{equation*}
\PP\left(\forall k,M\, \forall [\alpha,\beta]\in \mathcal{P}(M,2^{-k})\, \crsinf{\iota}{X_n}<\infty \right)=1
\end{equation*}
where $\mathcal{P}(M,l)$ is as in Definition \ref{def:P:subintervals}. Now applying Lemma \ref{res:bring:out:quantifiers}, noting that the inner formula is monotone decreasing in both $k$ and $M$, this is equivalent to
\begin{equation*}
\forall k,M \left(\PP\left(\forall [\alpha,\beta]\in \mathcal{P}(M,2^{-k})\, \crsinf{[\alpha,\beta]}{X_n}<\infty \right)=1\right)
\end{equation*}
and so for any $\alpha<\beta$, picking $k,M$ such that there exists $[\alpha',\beta']\in \mathcal{P}(M,2^{-k})$ with $[\alpha',\beta']\subseteq [\alpha,\beta]$ establishes that $\PP(\crsinf{[\alpha,\beta]}{X_n}<\infty)=1$. 
\begin{definition}
\label{def:quantitative:crossings}
Let $\seq{X_n}$ be a stochastic process.\medskip
\begin{enumerate}[(a)]

	\item For fixed $\alpha<\beta$, any function $\phi:(0,1)\to \RR$ satisfying
	\begin{equation*}
	\PP\left(\crsinf{[\alpha,\beta]}{X_n}\geq \phi(\lambda)\right)<\lambda \ \ \ \mbox{for all $\lambda\in (0,1)$}
	\end{equation*}
	is called a \emph{modulus of finite $[\alpha,\beta]$-crossings} for $\seq{X_n}$.\medskip
	
	\item Any function $\phi:(0,1)\times (0,\infty)\times \NN\to \RR$ satisfying
	\begin{equation*}
	\PP\left(\exists [\alpha,\beta]\in \mathcal{P}(M,l)\, \crsinf{[\alpha,\beta]}{X_n}\geq \phi(\lambda,M,l)\right)<\lambda
	\end{equation*}
	for all $\lambda\in (0,1)$, $M\in (0,\infty)$ and $l\in\NN$ is called a \emph{modulus of finite crossings} for $\seq{X_n}$.
\end{enumerate}
\end{definition}

\begin{lemma}
\label{res:quantitative:crossings}
\begin{enumerate}[(i)]

	\item If $\phi$ is a modulus of finite crossings and $\alpha<\beta$, then $\psi(\lambda):=\phi(\lambda,M,l)$ is a modulus of finite $[\alpha,\beta]$-crossings whenever $[\alpha',\beta']\subseteq [\alpha,\beta]$ for some $[\alpha',\beta']\in \mathcal{P}(M,l)$.\medskip
	
	\item If $\phi_{\alpha,\beta}$ is a modulus of finite $[\alpha,\beta]$-crossings for all $\alpha<\beta$, then
	\begin{equation*}
	\psi(\lambda,M,l):=\max\left\{\phi_{\alpha,\beta}\left(\frac{\lambda}{l}\right)\, \Bigl| \, [\alpha,\beta]\in \mathcal{P}(M,l)\right\}
	\end{equation*}
	is a modulus of finite crossings.

\end{enumerate}
\end{lemma}

\begin{proof}
The first part is clear, for (ii) we observe that
\begin{equation*}
\begin{aligned}
&\PP\left(\exists [\alpha,\beta]\in \mathcal{P}(M,l)\, \crsinf{[\alpha,\beta]}{X_n} 
\geq \psi(\lambda,M,l)\right)\\ 
&\leq \sum_{[\alpha,\beta]\in\mathcal{P}(M,l)}\PP\left(\crsinf{[\alpha,\beta]}{X_n}\geq \psi(\lambda,M,l)\right)\\
&\leq \sum_{[\alpha,\beta]\in\mathcal{P}(M,l)}\PP\left(\crsinf{[\alpha,\beta]}{X_n}\geq \phi_{\alpha,\beta}\left(\lambda/l\right)\right)\\
&<\sum_{[\alpha,\beta]\in\mathcal{P}(M,l)}\frac{\lambda}{l}=\lambda
\end{aligned}
\end{equation*}
where for the last step we recall that $\mathcal{P}(M,l)$ consists of $l$ intervals by definition.
\end{proof}

\begin{remark}
\label{rem:markov:crossings}
Just as in Lemma \ref{res:markov:tightness}, Markov's inequality gives us concrete moduli for the above when the expectation is bounded. For example, if $f:\RR\times \RR\to (0,\infty)$ is a function satisfying
\begin{equation*}
\EE(\crsinf{[\alpha,\beta]}{X_n})<\tau(\alpha,\beta)
\end{equation*}
for all $\alpha<\beta$, then $\phi_{[\alpha,\beta]}(\lambda):=\tau(\alpha,\beta)/\lambda$ is a modulus of finite $[\alpha,\beta]$-crossings for $\seq{X_n}$, and similarly for $\varepsilon$-fluctuations bounded in mean.
\end{remark}

\begin{remark}
\label{rem:upcr}
In the literature, one generally encounters inequalities that deal specifically with \emph{upcrossings} rather than crossings, as such inequalities are easier to prove directly. These typically take the form
\begin{equation*}
\EE(\upcrinf{[\alpha,\beta]}{X_n})\leq \tau(\alpha,\beta)
\end{equation*}
where the random variable $\upcrinf{[\alpha,\beta]}{X_n}$ represents the number of times $\seq{X_n}$ upcrosses the interval $[\alpha,\beta]$. Fixing some sequence $\seq{X_n(\omega)}$, it is clear that between any two consecutive upcrossings, there has to be exactly one downcrossing, and therefore
\begin{equation*}
\crsinf{[\alpha,\beta]}{X_n(\omega)}\leq 2\upcrinf{[\alpha,\beta]}{X_n(\omega)}+1
\end{equation*}
and thus
\begin{equation*}
\EE(\crsinf{[\alpha,\beta]}{X_n})\leq 2\EE(\upcrinf{[\alpha,\beta]}{X_n})+1\leq 2\tau(\alpha,\beta)+1.
\end{equation*}
Therefore, in general, upcrossing inequalities imply that $\seq{X_n}$ has almost surely finite crossings and give us the required modulus (the same is true for downcrossings, $\dcrinf{[\alpha,\beta]}{X_n}$). For us, it is slightly easier to work with crossing inequalities, and in concrete cases, these can always be instantiated from the relevant upcrossing or downcrossing inequalities.
\end{remark}

%%%%%%%%%%%%%%%%%%%%%%%%%%%%%%%%%%%%%%%%%%%%%%%%%%%%%%%%%%%%%%%%%%%%%%%%%%%%%%%%%%%%%%%%%%
\subsection{Almost sure convergence}
\label{sec:moduli:convergence}
%%%%%%%%%%%%%%%%%%%%%%%%%%%%%%%%%%%%%%%%%%%%%%%%%%%%%%%%%%%%%%%%%%%%%%%%%%%%%%%%%%%%%%%%%%

By encoding quantification over $\varepsilon\in (0,1)$ by quantification over $2^{-k}$, entirely analogously to the case of fluctuations, we can show that almost sure Cauchy convergence of a stochastic process $\seq{X_n}$ is equivalent to:
\begin{equation*}
\forall \varepsilon\in (0,1)\, \PP(\exists n\, \forall m\geq n\, \forall i,j\in [n;m](|X_i-X_j|< \varepsilon))=1
\end{equation*}
This has exactly the form of Theorem \ref{res:equivalent:metastabe} (a) for
\[
A(n,m):=\exists i,j\in [n;m]\, (|X_i-X_j|\geq \varepsilon)
\]
where $A(n,m)$ satisfies the required monotonicity property that $A(n',m')\subseteq A(n,m)$ for $n\leq n'$ and $m'\leq m$, and also $A(n,m)=\emptyset$ for $m\leq n$. Therefore the results of Section \ref{sec:finitizing:metastability} and \ref{sec:finitizing:learnability} apply in this case, and we arrive at the following definitions:

\begin{definition}
\label{def:quantitative:almostsureconvergence}
Let $\seq{X_n}$ be a stochastic process.\medskip
\begin{enumerate}[(a)]

	\item Any function $\phi:(0,1)\times (0,1)\to \RR$ satisfying
	\begin{equation*}
	\PP(\exists i,j\geq \phi(\lambda,\varepsilon)\, (|X_i-X_j|\geq \varepsilon))<\lambda
	\end{equation*}
	for all $\lambda,\varepsilon\in(0,1)$ is called a \emph{rate of almost sure convergence} for $\seq{X_n}$.\medskip
	
	\item Any functional $\Phi:(0,1)\times (0,1)\times (\NN\to\NN)\to \NN$ such that for all $\lambda,\varepsilon\in (0,1)$ and $g:\NN\to\NN$ there exists $n\leq \Phi(\lambda,\varepsilon,g)$ satisfying
	\begin{equation*}
	\PP(\exists i,j\in [n;n+g(n)](|X_i-X_j|\geq \varepsilon))<\lambda
	\end{equation*}
	is called a \emph{metastable rate of uniform convergence} for $\seq{X_n}$.\medskip

	\item Any functional $\Phi:(0,1)\times (0,1)\times (\NN\to\NN)\to \NN$ such that for all $\lambda,\varepsilon\in(0,1)$ and $g:\NN\to\NN$
	\begin{equation*}
	\PP(\forall n\leq \Phi(\lambda,\varepsilon,g)\, \exists i,j\in [n;n+g(n)](|X_i-X_j|\geq \varepsilon))<\lambda
	\end{equation*}
	is called a \emph{metastable rate of pointwise convergence} for $\seq{X_n}$.
	
\end{enumerate}
\end{definition}
None of these definitions are new: The former is nothing more than a convergence rate for
\begin{equation*}
\lim_{n\to\infty}\PP\left(\sup_{i,j\geq n}(|X_i-X_j|\geq \varepsilon)\right)=0
\end{equation*}
which is just a Cauchy variant of the obvious way of formulating the speed of almost sure convergence, while the two metastable versions of this property are introduced (for general measures) in \cite{avigad-dean-rute:12:dominated}, where in the terminology used in \cite{avigad-dean-rute:12:dominated}, a uniform metastable rate generates bounds on the \emph{$\lambda$-uniform $\varepsilon$-metastable convergence of $\seq{X_n}$} for all $\lambda,\varepsilon>0$, while a pointwise rate generates a \emph{$\lambda$-uniform bound for the $\varepsilon$-
metastable pointwise convergence of $\seq{X_n}$} for all $\lambda,\varepsilon>0$.

\begin{remark}
\label{rem:specker}
The main reason for introducing and studying metastable rates of almost sure convergence is that direct computable rates are, in general, impossible for the classes of stochastic processes we consider. A standard result of computable analysis due to Specker \cite{specker:49:sequence} asserts the existence of monotone sequences $\seq{s_n}$ of rational numbers with no computable limit, and hence no computable rates of convergence. Any monotone sequence of rational numbers can be trivially encoded as a sub- or supermartingale where each element is constant, i.e.\ $X_n:=s_n$ for all $n\in\NN$, and since then we would have
\[
\PP(\exists i,j\geq n\, |X_i-X_j|\geq \varepsilon)<\frac{1}{2}\iff \forall i,j\geq n\, |s_i-s_j|<\varepsilon
\]
a computable solution to the halting problem would result from a computable rate of almost sure convergence for $\seq{X_n}$. We emphasise that this situation is in stark contrast to related work in computable analysis \cite{rute:etal:algorithmic:doob}, which imposes computability assumptions (e.g. on the elements and limit of $\seq{X_n}$) to obtain almost sure rates.
\end{remark}

We already appealed to an abstract construction from \cite{avigad-dean-rute:12:dominated} in our proof of Theorem \ref{res:equivalent:metastabe}. In fact, the main goal of \cite{avigad-dean-rute:12:dominated} was to apply this abstract construction to provide a route from pointwise to uniform metastable rates of convergence, and in this way give a quantitative account of Egorov's theorem (which in particular implies a strengthening of a quantitative dominated convergence theorem formulated and used by Tao \cite[Theorem A.2]{tao:08:ergodic}). Their central result is as follows:
\begin{theorem}
[Avigad, Dean \& Rute, Theorem 3.1 of \cite{avigad-dean-rute:12:dominated}]
\label{res:egorov:avigad}
For every $\varepsilon>0$, $\lambda>\lambda'>0$ and functional $M_1:(\NN\to\NN)\to \NN$, there is a functional 
\begin{equation*}
M_2:=\Gamma(\lambda,\lambda',M_1)
\end{equation*}
such that whenever
\begin{equation}
\label{eqn:pointwise:avigad}
\PP(\forall n\leq M_1(f_1)\, \exists i,j\in [n;f_1(n)](|X_i-X_j|\geq \varepsilon))<\lambda'
\end{equation}
for all $f_1:\NN\to\NN$, then for any $f_2:\NN\to\NN$ there exists some $n\leq M_2(f_2)$ such that
\begin{equation}
\label{eqn:uniform:avigad}
\PP(\exists i,j\in [n;f_2(n)](|X_i-X_j|\geq \varepsilon))<\lambda.
\end{equation}
\end{theorem}
As a direct consequence of Theorem \ref{res:egorov:avigad} we obtain the following passage from metastable pointwise to uniform rates:
\begin{corollary}
\label{res:pointwise:to:uniform}
Suppose that $\Phi$ is a metastable rate of pointwise convergence for $\seq{X_n}$ and let $\Gamma$ be the construction from Theorem \ref{res:egorov:avigad}. Then, a metastable rate of uniform convergence is given by
\begin{equation*}
\Psi(\lambda,\varepsilon,g):=\Gamma\left(\lambda,\tfrac{\lambda}{2},M_1^{\lambda,\varepsilon}\right)(\tilde g)
\end{equation*}
where $\tilde g(n):=n+g(n)$ and 
\begin{equation*}
M_1^{\lambda,\varepsilon}(f_1):=\Phi\left(\tfrac{\lambda}{2},\varepsilon,\bar{f}_1\right)
\end{equation*}
for $\bar{f}_1(n):=f_1(n)-n$ if $f_1(n)\geq n$ and $\bar{f}_1(n):=0$ otherwise.
\end{corollary}

\begin{proof}
Fixing $\lambda,\varepsilon>0$, the functional $M_1^{\lambda,\varepsilon}$ satisfies (\ref{eqn:pointwise:avigad}) for $\lambda':=\lambda/2$, noting that $n+\bar{f}_1(n)=f_1(n)$ unless $f_1(n)<n$, in which case $\exists i,j\in [n;f_1(n)](|X_i-X_j|\geq \varepsilon)$ and $\exists i,j\in [n;n+\bar{f}_1(n)](|X_i-X_j|\geq \varepsilon)$ are both empty. Thus by Theorem \ref{res:egorov:avigad}, for any $g:\NN\to\NN$ there exists some $n\leq \Psi(\lambda,\varepsilon,g)$ satisfying (\ref{res:pointwise:to:uniform}) for $f_2:=\tilde{g}$, and since $\lambda,\varepsilon>0$ are arbitrary, we are done.
\end{proof}

\begin{remark}
\label{rem:blowup}
The precise construction of $\Gamma$ in \cite{avigad-dean-rute:12:dominated} uses a form of recursion on trees known as bar recursion (introduced in \cite{spector:62:barrecursion}). In general, this causes a phenomenal growth in complexity: Even the trivial input sketched on page 11 of \cite{avigad-dean-rute:12:dominated} results in a tower of exponentials when passed through $\Gamma$. However, we will see that in cases where we are able to explicitly define pointwise metastable rates of convergence, strengthening our assumptions, typically to some bound on the expected value of fluctuations or crossings, allows us to find uniform metastable rates of similarly low complexity, demonstrating that this apparent blowup in complexity from pointwise to uniform convergence may not actually be present in ordinary mathematical situations, such as martingale convergence.
\end{remark}

In this paper, all pointwise or uniform metastable rates we obtain will have learnable form in the sense of Definition \ref{def:learnability}, which we now make explicit as follows:

\begin{definition}
\label{def:quantitative:almostsureconvergence:learnable}
Let $\seq{X_n}$ be a stochastic process.\medskip
\begin{enumerate}[(a)]

	\item Any function $\phi:(0,1)\times (0,1)\to \RR$ satisfying
    \[
    \exists n\leq \phi(\lambda,\varepsilon)\, \PP\left(\exists i,j\in [a_n;b_n]\, (|X_i-X_j|\geq \varepsilon)\right)<\lambda
    \]
    for any $a_0<b_0\leq a_1<b_1\leq \ldots$ is called a \emph{learnable rate of uniform convergence}.\medskip

	\item Any function $\phi:(0,1)\times (0,1)\to \RR$ satisfying
    \[
    \PP\left(\forall n\leq \phi(\lambda,\varepsilon)\, \exists i,j\in [a_n;b_n]\, (|X_i-X_j|\geq \varepsilon)\right)<\lambda
    \]
    for any $a_0<b_0\leq a_1<b_1\leq \ldots$ is called a \emph{learnable rate of pointwise convergence}.
	
\end{enumerate}
\end{definition}
By Lemma \ref{res:learnable:equivalences}, learnable rates of uniform convergence $\phi(\lambda,\varepsilon)$ correspond to metastable rates of uniform convergence of the form $\Phi(\lambda,\varepsilon,g)=\tilde g^{(\lceil\phi(\lambda,\varepsilon)\rceil)}(0)$ and similarly for pointwise convergence.

%%%%%%%%%%%%%%%%%%%%%%%%%%%%%%%%%%%%%%%%%%%%%%%%%%%%%%%%%%%%%%%%%%%%%%%%%%%%%%%%%%%%%%%%%%
\subsection{From fluctuations to convergence}
\label{sec:moduli:fluctoconv}
%%%%%%%%%%%%%%%%%%%%%%%%%%%%%%%%%%%%%%%%%%%%%%%%%%%%%%%%%%%%%%%%%%%%%%%%%%%%%%%%%%%%%%%%%%

Instantiating Theorem \ref{res:abstract:fluctuations:to:rates} in the concrete setting of $\varepsilon$-fluctuations and almost sure Cauchy convergence gives us our first simple quantitative convergence results:

\begin{theorem}
    \label{res:fluc:to:convergence}

    Let $\seq{X_n}$ be a stochastic process.
    
    \begin{enumerate}[(i)]

        \item If $\phi(\lambda,\varepsilon)$ is a modulus of fluctuations for $\seq{X_n}$, then it is also a learnable rate of pointwise convergence, and thus $\Phi(\lambda,\varepsilon,g)={g^{(\lceil\phi(\lambda,\varepsilon)\rceil)}(0)}$ is a metastable rate of pointwise convergence for $\seq{X_n}$.\medskip

        \item If $\tau(\varepsilon)$ such that $\EE[\flucinf{\varepsilon}{X_n}]\leq \tau(\varepsilon)$ for all $\varepsilon\in (0,1)$, then $\phi(\lambda,\varepsilon):= \tau(\varepsilon)/\lambda$ is a learnable rate of uniform convergence, and thus $\Phi(\lambda,\varepsilon,g)={\tilde g^{(\lceil\tau(\varepsilon)/\lambda\rceil)}(0)}$ is a metastable rate of uniform convergence for $\seq{X_n}$.

    \end{enumerate}
\end{theorem}   

We now give two elementary examples that explore the relationship between the two kinds of learnable rates and fluctuations.

\begin{example}
\label{ex:finitefluc:uniformfluc}
Define $\seq{X_n}$ on {the standard Borel measure space} $([0,1],\mathcal{F},\mu)$ by
    \begin{equation*}
        X_{2n}:=0 \ \ \ X_{2n+1}:=I_{(1/2(n+1),1/(n+1))}.
    \end{equation*}
   For fixed $\omega\in [0,1]$, we have $X_n(\omega)=0$ for sufficiently large $n$ (depending on $\omega$), and therefore $\seq{X_n}$ is almost surely convergent to zero and thus has finite fluctuations almost surely. More precisely, for $\varepsilon\in (0,1)$, every $n$ such that $X_{2n+1}(\omega)=1$ gives rise to two $\varepsilon$-fluctuations, and all such fluctuations arise in this way, so letting $\# S$ denote the number of elements in a finite set $S$, we have
    \begin{equation*}
        \flucinf{\varepsilon}{X_n(\omega)}\leq 2\#\{n \, | \, \omega\in I_{(1/2(n+1),1/(n+1))}\}\leq 2\#\{n \, | \, \omega< 1/(n+1)\}\leq 2/\omega
    \end{equation*}
    so that
    \begin{equation*}
        \PP(\flucinf{\varepsilon}{X_n}\geq a)\leq \PP(\{\omega\, | \, 2/\omega\geq a\})=2/a.
    \end{equation*}
    Therefore, $\phi(\lambda,\varepsilon):=1+2/\lambda$ is a modulus of fluctuations for all $\varepsilon\in (0,1)$, and thus also a learnable rate of pointwise convergence. On the other hand, for $\varepsilon\in (0,1)$ we have
    \begin{equation*}
    \begin{aligned}
        \PP(\exists i,j\in [2n;2n+1](|X_i-X_j|\geq\varepsilon))=\PP(X_{2n+1}=1)=\frac{1}{2(n+1)}
    \end{aligned}
    \end{equation*}
    and thus
    \begin{equation*}
       \EE[\flucinf{\varepsilon}{X_n}]\geq \sum_{n=0}^\infty \PP(\exists i,j\in [2n;2n+1](|X_i-X_j|\geq \varepsilon))=\sum_{n=0}^\infty \frac{1}{2(n+1)}=\infty
    \end{equation*}
    However, in this case, learnable rates of uniform convergence (and so, learnable metastable rates of uniform convergence) can be found. For example, we have
    \[
    \exists n\leq (1/2\lambda) \, \PP(\exists i,j\in [2n;2n+1](|X_i-X_j|\geq \varepsilon))<\lambda
    \]
    and it can be shown more generally (by taking into account the overlap of intervals) that a learnable uniform rate is given by 
    \begin{equation*}
        \psi(\lambda,\varepsilon):=   2/\lambda
    \end{equation*}
\end{example}

\begin{example}
\label{ex:uniform:stronger:pointwise}
Let $\mathcal{C}$ be the class of nonnegative stochastic processes $\seq{X_n}$ that are monotone and uniformly bounded above by $1$. These can experience at most $1/\varepsilon$ $\varepsilon$-fluctuations, and therefore, a modulus of fluctuations and hence learnable rate of pointwise convergence for any such process is given by
\[
\phi(\lambda,\varepsilon)= \frac{1}{\varepsilon} +1.
\]
Suppose now that $\psi(\lambda,\varepsilon)$ is a learnable rate of uniform convergence that applies uniformly in $\mathcal{C}$ i.e. $\psi(\lambda,\varepsilon)$ is a learnable rate of uniform convergence for any $\seq{X_n}\in\mathcal{C}$. Then we claim that
\begin{equation}
\label{eqn:uniformbound}
\frac{1}{\lambda\varepsilon}\leq \psi(\lambda,\varepsilon)
\end{equation}
for all $\varepsilon,\lambda\in (0,1)$. Suppose for contradiction that there exist $\varepsilon,\lambda\in (0,1)$ on which (\ref{eqn:uniformbound}) fails, where for simplicity we assume that $\varepsilon=1/M$ and $\lambda=1/N$ for some $M,N\in\NN$. We define a stochastic process $\seq{X_n}$ on the standard space $([0,1],\mathcal{F},\mu)$ and in terms of these parameters as follows: First, define the sequence of reals $\seq{x_n}$ by
\begin{equation*}
x_n:=\begin{cases}
  0 & \mbox{if $n=0$}\\
	i/M & \mbox{if $(i-1)N< n\leq iN$ for $i=1,\ldots,M$}\\
	1 & \mbox{if $n>MN$}
\end{cases}
\end{equation*}
so that we have $x_{j+1}-x_{j}=1/M$ for $j=iN$ and $i=0,\ldots,M-1$, and $x_{j+1}-x_{j}=0$ for all other $j\in\NN$. Now letting $I_0,\ldots,I_{N-1}$ represent a division of $[0,1]$ into $N$ equal partitions, we define 
\begin{equation*}
X_n(\omega):=\begin{cases}0 & \mbox{if $n<k$}\\ x_{n-k} & \mbox{otherwise}\end{cases} \ \ \ \mbox{for $\omega\in I_k$.}
\end{equation*}
Then analogously to the situation with $\seq{x_n}$, for $k=0,\ldots,N-1$ and $\omega\in I_k$ we have $X_{j+1}(\omega)-X_{j}(\omega)=1/M$ for $j=iN+k$ and $i=0,\ldots,M-1$, and $X_{j+1}(\omega)-X_{j}(\omega)=0$ for all other $j\in\NN$. This means that for all $j\leq MN-1$, there is exactly one $k=0,\ldots,N-1$ for which $X_{j+1}-X_{j}=1/M$ on $I_k$, and therefore
\begin{equation*}
\PP(|X_j-X_{j+1}|\geq \varepsilon)=\lambda
\end{equation*}
for all $\forall j\leq (1/\lambda\varepsilon)-1$. But since (\ref{eqn:uniformbound}) fails, we must have 
\[
\exists j<(1/\lambda\varepsilon)\, \PP(|X_j-X_{j+1}|\geq \varepsilon)<\lambda
\]
contradicting that this is a learnable rate of uniform convergence for $\seq{X_n}$, and thus a rate that applies to all sequences in $\mathcal{C}$. 
\end{example}

%%%%%%%%%%%%%%%%%%%%%%%%%%%%%%%%%%%%%%%%%%%%%%%%%%%%%%%%%%%%%%%%%%%%%%%%%%%%%%%%%%%%%%%%%%
\subsection{Outline of next sections}
\label{sec:modulo:summary}
%%%%%%%%%%%%%%%%%%%%%%%%%%%%%%%%%%%%%%%%%%%%%%%%%%%%%%%%%%%%%%%%%%%%%%%%%%%%%%%%%%%%%%%%%%

While Theorem \ref{res:fluc:to:convergence} gives us a very direct route from bounds on fluctuations to rates of convergence, it is bounds on crossings that play a more significant role in the theory or martingales (and also related areas such as ergodic theory). Thus, we are more interested in establishing a route from analogous bounds on crossings to rates of convergence. This is more subtle and forms the content of the next two sections, which examine the pointwise (Section \ref{sec:pointwise}) and uniform (Section \ref{sec:uniform}) cases separately. Then, in Section \ref{sec:applications}, we focus in greater detail on applications which arise from the uniform case.

%%%%%%%%%%%%%%%%%%%%%%%%%%%%%%%%%%%%%%%%%%%%%%%%%%%%%%%%%%%%%%%%%%%%%%%%%%%%%%%%%%%%%%%%%%
%%%%%%%%%%%%%%%%%%%%%%%%%%%%%%%%%%%%%%%%%%%%%%%%%%%%%%%%%%%%%%%%%%%%%%%%%%%%%%%%%%%%%%%%%%
\section{From finite crossings to pointwise convergence}
\label{sec:pointwise}
%%%%%%%%%%%%%%%%%%%%%%%%%%%%%%%%%%%%%%%%%%%%%%%%%%%%%%%%%%%%%%%%%%%%%%%%%%%%%%%%%%%%%%%%%%
%%%%%%%%%%%%%%%%%%%%%%%%%%%%%%%%%%%%%%%%%%%%%%%%%%%%%%%%%%%%%%%%%%%%%%%%%%%%%%%%%%%%%%%%%%

We begin by providing a stochastic analogue of Proposition \ref{res:fluctuations:crossings}, which, for any stochastic process, provides a quantitative route from the properties of being uniformly bounded and having finite crossings to having finite fluctuations (and hence a learnable rate of pointwise convergence). Results of this kind have been given before, notably in \cite{kachurovskii:96:convergence} (see e.g. Theorem 27 of that paper), but our use of abstract moduli allows us to give them a more general form:
\begin{theorem}
\label{res:fluctuations:crossings:stochastic}
Let $\seq{X_n}$ be a stochastic process. If $\phi$ is a modulus of finite crossings and $f$ is a modulus of uniform boundedness for $\seq{X_n}$ then
	\begin{equation*}
	\psi(\lambda,\varepsilon):=l\cdot\phi\left(\frac{\lambda}{2},M,l\right) \ \ \ \mbox{for} \ \  l:=\Bigl\lceil\frac{4M}{\varepsilon}\Bigr\rceil \ \ \mbox{and} \ \ M:=f\left(\frac{\lambda}{2}\right)
	\end{equation*}
	is a modulus of finite fluctuations for the same process, and therefore also a learnable rate of pointwise convergence.
\end{theorem}

\begin{proof}
Fix $\lambda>0$ and note that for any event $A$ and $M:=f(\lambda/2)$:
\begin{equation*}
\begin{aligned}
\PP(A)&\leq \PP\left(\sup_{n\in\NN}|X_n|> M\right)+\PP\left(A\cap \sup_{n\in\NN}|X_n|\leq M\right)\\
&<\frac{\lambda}{2}+\PP\left(A\cap \sup_{n\in\NN}|X_n|\leq M\right).
\end{aligned}
\end{equation*}
Therefore it suffices if for any $\varepsilon>0$ we can find $N\in\NN$ such that
\begin{equation}
\label{eqn:fluc:sup}
\PP\left(\flucinf{\varepsilon}{X_n}\geq N\cap \sup_{n\in\NN}|X_n|\leq M\right)<\frac{\lambda}{2}.
\end{equation}
For any fixed $\omega\in\Omega$, reasoning as in the proof of Proposition \ref{res:fluctuations:crossings}, if $\flucinf{\varepsilon}{X_n(\omega)}\geq N$ and $\sup_{n\in\NN}|X_n(\omega)|\leq M$, then for $l:=\lceil 4M/\varepsilon\rceil$, any interval in $\mathcal{P}(M,l)$ has width $\leq \varepsilon/2$, and so any $\varepsilon$-fluctuation of $\seq{X_n(\omega)}$ is also an $[\alpha,\beta]$-crossing for some $[\alpha,\beta]\in \mathcal{P}(M,l)$. By the pigeonhole principle there must therefore be some $[\alpha,\beta]\in \mathcal{P}(M,l)$ with $\crsinf{[\alpha,\beta]}{X_n(\omega)}\geq N/l$, and so we have shown that
\begin{equation*}
\PP\left(\flucinf{\varepsilon}{X_n}\geq N\cap \sup_{n\in\NN}|X_n|\leq M\right)\subseteq \PP\left(\exists [\alpha,\beta]\in \mathcal{P}(M,l)\, \crsinf{[\alpha,\beta]}{X_n(\omega)}\geq \frac{N}{l}\right)
\end{equation*}
Thus (\ref{eqn:fluc:sup}) holds for $N:=l\cdot\phi(\lambda/2,M,l)$.
\end{proof}

We now give a rephrasing of Theorem \ref{res:fluctuations:crossings:stochastic} in terms of traditional rates rather than our proof-theoretic moduli (c.f. the discussion at the beginning of Section \ref{sec:moduli}). This both facilitates comparison with known results on fluctuation bounds and also demonstrates the compact way in which moduli both encode and allow us to manipulate convergence rates. We note that the derivation of this result from Theorem \ref{res:fluctuations:crossings:stochastic} (ii) involves nothing beyond the manipulation of functions. For a similar translation from a quantitative result in terms of moduli to one involving ordinary rates, see \cite[Lemma 3.6]{powell-wiesnet:21:contractive}.

\begin{corollary}
\label{res:crossings:fluctuations:concrete}
Let $\seq{X_n}$ be a stochastic process such that
\begin{enumerate}[(a)]

	\item $\PP\left(\sup_{n\in\NN}|X_n|\geq a\right)<g(a)$ for all $a>0$, where $g$ is a strictly decreasing function satisfying $g(a) \to 0$ as $a \to \infty$,\smallskip
	
	\item $\PP\left(\crsinf{[\alpha,\beta]}{X_n} \ge a\right)< h_{\alpha,\beta}(a)$ for all $\alpha<\beta$ such that $\PP\left(\crsinf{[\alpha,\beta]}{X_n} > 0\right)>0$ and $a > 0$, where $h_{\alpha,\beta}$ is a strictly decreasing function satisfying $h_{\alpha,\beta}(a) \to 0$ as $a \to \infty$.

\end{enumerate}\smallskip
Then for all $\varepsilon>0$
\begin{equation*}
\PP\left(\flucinf{\varepsilon}{X_n}\geq a\right)<G^{-1}_\varepsilon(a)
\end{equation*}
for any strictly decreasing function, $G_\varepsilon$ satisfying
\begin{equation*}
l\cdot H\left(\frac{\lambda}{2},g^{-1}\left(\frac{\lambda}{2}\right),l\right)< G_\varepsilon(\lambda) \ \ \ \mbox{for} \ \ \ l:=\Bigl\lceil \frac{4g^{-1}(\lambda/2)}{\varepsilon}\Bigr\rceil
\end{equation*}
where $H$ is any function such that
\begin{equation*}
h^{-1}_{\alpha,\beta}\left(\frac{\lambda}{l}\right)\leq H(\lambda,M,l)
\end{equation*}
for any $\lambda,M,l>0$ and $[\alpha,\beta]\in\mathcal{P}(M,l)$ with $\PP\left(\crsinf{[\alpha,\beta]}{X_n} > 0\right)>0$.
\end{corollary}

\begin{proof}
By definition, $g^{-1}$ is a modulus of uniform boundedness for $\seq{X_n}$, and $h^{-1}_{\alpha,\beta}$ is a modulus of finite $[\alpha,\beta]$-crossings for all $\alpha<\beta$ with $\PP\left(\crsinf{[\alpha,\beta]}{X_n} > 0\right)>0$. By Lemma \ref{res:quantitative:crossings} and the property of $H$ (noting that the restriction to $[\alpha,\beta]\in\mathcal{P}(M,l)$ with $\PP\left(\crsinf{[\alpha,\beta]}{X_n} > 0\right)>0$ does not affect Lemma \ref{res:quantitative:crossings}), $H$ must be a modulus of finite crossings for $\seq{X_n}$. Now, by Theorem \ref{res:fluctuations:crossings:stochastic} (ii), any bound on
\begin{equation*}
l\cdot H(\lambda/2,g^{-1}(\lambda/2),l) 
\end{equation*}
for $l:=\lceil 4g^{-1}(\lambda/2)/\varepsilon\rceil$ is a modulus of finite $\varepsilon$-fluctuations for $\seq{X_n}$, and so by definition we have
\[
\PP\left(\flucinf{\varepsilon}{X_n}\geq G_\varepsilon(\lambda)\right)<\lambda
\]
for all $\lambda>0$, from which the result follows.
\end{proof}

\begin{example}
\label{ex:kachurovskii}
In the special case that $\seq{X_n}$ satisfies 
\begin{enumerate}[(i)]

	\item $\PP\left(\sup_{n\in\NN}|X_n|\geq a\right)<\frac{S}{a}$ for all $a>0$, \smallskip
	
	\item $\EE\left(\upcrinf{[\alpha,\beta]}{X_n}\right)< \frac{S+|\alpha|}{\beta-\alpha}$ for all $\alpha<\beta$,

\end{enumerate}
applying Corollary \ref{res:crossings:fluctuations:concrete} gives us
\begin{equation*}
\PP\left(\flucinf{\varepsilon}{X_n}\geq a\right)<\frac{c}{a^{1/4}}\left(1+\frac{S}{\varepsilon}\right)
\end{equation*}
for a constant $c\leq 3\times 16^3$ (we omit the details of this calculation). This case is already proven by Kachorovskii as \cite[Theorem 27]{kachurovskii:96:convergence}, where a better value of $c=7$ is obtained, though this requires some more careful (and ad-hoc) calculations.
\end{example} 
We now present a more involved example, where our abstract framework and use of moduli seems to result in a slight improvement of a known bound.
\begin{example}
\label{ex:ivanov}
Fix a measure preserving transformation $\tau:\Omega\to\Omega$ of our probability space $X:=(\Omega,\mathcal{F},\PP)$ and define the Koopman operator $T:L_1(X)\to L_1(X)$ by $Tf:=f\circ\tau$. For $f\in L_1(X)$ we define
\begin{equation*}
\begin{aligned}
S_nf:=\sum_{k=0}^{n-1} T^kf \ \  \mbox{and} \ \  A_nf:=\frac{S_nf}{n}.
\end{aligned}
\end{equation*}
The Birkhoff pointwise ergodic theorem states that the ergodic averages $A_nf$ converge almost surely, and it is shown by Kachorovskii in \cite[Theorem 23]{kachurovskii:96:convergence} (though the result is attributed to Ivanov) that the following bound on the probabilistic fluctuations can be given: 
\begin{equation}
\label{eqn:ivanov}
\PP\left(\flucinf{\varepsilon}{A_nf}\geq a\right)<c\sqrt{\frac{\log\left(a\right)}{a}}
\end{equation}
for all $f \in L_1(X)$ and $\varepsilon, a >0$, with $c > 0$ a constant that depends on $\EE(|f|)/\varepsilon$.

An alternative bound can be obtained through Corollary \ref{res:crossings:fluctuations:concrete}. As in \cite{kachurovskii:96:convergence}, we first assume $f\ge 0$. By the maximal ergodic theorem, we have for all $a > 0$,
\begin{equation*}
    \PP\left(\sup_{n\in\NN}|A_nf|\geq a\right)\le\frac{\EE(|f|)}{a},
\end{equation*}
so for any $S >\EE(|f|)$ we can take $g(a):=S/a$ in Corollary \ref{res:crossings:fluctuations:concrete}. For crossings, we use the well-known result of Ivanov \cite{Ivanov:oscillations:96} which states that for $0<\alpha<\beta$ and $k>0$,
\begin{equation*}
    \PP\left(\dcrinf{[\alpha,\beta]}{A_nf} \ge k\right)\le  \left(\frac{\alpha}{\beta}\right)^k 
\end{equation*}
where $\dcrinf{[\alpha,\beta]}{A_nf}$ denotes the number of downcrossings of $[\alpha,\beta]$ made by $\seq{A_nf}$. Thus Remark \ref{rem:upcr} allows us to conclude
\begin{equation*}
    \PP\left(\crsinf{[\alpha,\beta]}{A_nf} \ge k\right)<  \left(\frac{\alpha}{\beta}\right)^\frac{k-1}{4}
\end{equation*}
(we divide the exponent by another factor of 2 to obtain a strict inequality), so we can take  $h_{\alpha,\beta}(k) =(\alpha/\beta)^{\frac{k-1}{4}}$ in Corollary \ref{res:crossings:fluctuations:concrete}. We can restrict our attention $0< \alpha < \beta$ (i.e. the situation $\PP\left(\crsinf{[\alpha,\beta]}{A_nf}>0\right)>0$), and in this case
\[
h^{-1}_{\alpha,\beta}(\lambda)=\frac{4\log(1/\lambda)}{\log(\beta)-\log(\alpha)}+1\leq \frac{4\beta\log(1/\lambda)}{\beta-\alpha}+1
\]
where for the last step, we use
\[
\log(\beta) - \log(\alpha) \ge \frac{\beta -\alpha}{\beta}, 
\]
which follows from the mean value theorem. Thus for any $M,l>0$ and $[\alpha,\beta]\in \mathcal{P}(M,l)$ with $0<\alpha<\beta$, we have
\[
h^{-1}_{\alpha,\beta}\left(\frac{\lambda}{l}\right)\leq \frac{4\beta\log(l/\lambda)}{\beta-\alpha}+1=2l\cdot \log\left(\frac{l}{\lambda}\right)+1\leq 2l\cdot \log\left(\frac{2l}{\lambda}\right)=:H(\lambda,M,l)
\] 
and so the right hand side defines a suitable bounding function $H$. We then observe that for 
\[
l=\Bigl\lceil\frac{4g^{-1}(\lambda/2)}{\varepsilon}\Bigr\rceil\leq \frac{9S}{\lambda\varepsilon}
\]
it follows that
\[
l\cdot H\left(\frac{\lambda}{2},g^{-1}\left(\frac{\lambda}{2}\right),l\right)=2l^2\cdot \log\left(\frac{4l}{\lambda}\right)\leq c\left(\frac{S}{\lambda\varepsilon}\right)^2\cdot \log\left(\frac{cS}{\lambda^2\varepsilon}\right)=:G_{\varepsilon}(\lambda)
\]
for suitable constant $c\leq 200$. It remains to find the inverse of the function $G_{\varepsilon}$ defined above. To this end, suppose that
\[
a=c\left(\frac{S}{\lambda\varepsilon}\right)^2\cdot \log\left(\frac{cS}{\lambda^2\varepsilon}\right).
\] 
Rearranging we obtain
\[
\exp\left[\frac{a}{c}\left(\frac{\lambda\varepsilon}{S}\right)^2\right]=\frac{cS}{\lambda^2\varepsilon}.
\]
Now letting $E(x):=x\exp(x)$ we have
\[
E\left[\frac{a}{c}\left(\frac{\lambda\varepsilon}{S}\right)^2\right]=\frac{a}{c}\left(\frac{\lambda\varepsilon}{S}\right)^2\cdot \frac{cS}{\lambda^2\varepsilon}=\frac{aS}{\varepsilon}
\]
and therefore
\[
\lambda=\frac{S}{\varepsilon}\cdot\sqrt{\frac{c}{a}\cdot W\left(\frac{aS}{\varepsilon}\right)}.
\]
Where, $W$ is the inverse of the $E$ (i.e. the Lambert $W$-function), and so by Corollary \ref{res:crossings:fluctuations:concrete} we have
\[
\PP\left(\flucinf{\varepsilon}{A_nf}\geq a\right)<c_0\sqrt{\frac{W\left(c_0a\right)}{a}}
\]
for $c_0:=S\sqrt{c}/\varepsilon$. This implies Ivanov's bound (\ref{eqn:ivanov}) and improves it slightly in that $W(c_0a)<\log(a)$ for $c_0<\log(a)$. To see that the same bound holds for general $f$, not assumed to be nonnegative (for different constant $c_0$) we can decompose $f$ as the difference of two positive terms $f = f^+ -f^-$ where $f^+:= \max\{f, 0\}$ and $f^-:= \max\{-f, 0\}$ are the positive and negative parts of $f$ respectfully. Furthermore, we will have $\EE(|f^\pm|) < S$, $A_nf= A_n(f^+) - A_n(f^-)$ and 
    \begin{equation*}
    \PP\left(\flucinf{\varepsilon/2}{A_nf}\geq a\right) \le \PP\left(\flucinf{\varepsilon/2}{A_n(f^+)}\geq a/2\right) + \PP\left(\flucinf{\varepsilon}{A_n(f^-)}\geq a/2\right).
    \end{equation*}
Although we did not introduce any fundamentally new ideas to the proof of (\ref{eqn:ivanov}) given in \cite{kachurovskii:96:convergence}, our approach allows us to be more precise: The improved bound cannot be easily obtained through adapting the proof of \cite[Theorem 23 ]{kachurovskii:96:convergence}. Finally, this example shows that the apparently distinct results \cite[Theorems 23 and 27]{kachurovskii:96:convergence} are both subsumed under our abstract framework.
\end{example}

An open question at this stage is whether Theorem \ref{res:fluctuations:crossings:stochastic} (ii) can be improved by replacing the modulus of uniform boundedness with a modulus of tightness while still obtaining an explicit modulus of fluctuations in terms of that modulus. It is certainly the case that we can replace uniform boundedness with the weaker property of tightness in order to prove that finite crossings implies finite fluctuations: To see this, one can use a pointwise argument, namely that if $\crsinf{[\alpha,\beta]}{X_n(\omega)}<\infty$ for all $\alpha<\beta$, then $\seq{X_n(\omega)}$ converges to a limit in $\RR\cup\{\pm\infty\}$, and therefore 
\begin{equation*}
X_\infty:=\lim_{n\to\infty}X_n
\end{equation*}
exists almost surely in $\RR\cup\{\pm\infty\}$. By tightness of $\seq{X_n}$ we have that for any $\lambda>0$ there exists $N\in\NN$ such that
\begin{equation*}
\PP\left(|X_\infty|\geq N\right)=\PP\left(\liminf_{n\to\infty}|X_n|\geq N\right)\leq \liminf_{n\to\infty}\PP\left(|X_n|\geq N\right)< \lambda
\end{equation*}
by Fatou's lemma, and therefore $|X_\infty|<\infty$ almost surely, which in turn implies that $\seq{X_n}$ converges and thus has finite fluctuations almost surely. However, converting this argument into a quantitative one, where one obtains a concrete modulus of finite fluctuations in terms of the corresponding moduli of tightness and finite crossing, is less obvious, as Fatou's lemma seems fundamentally nonconstructive.

We conclude this section with an example that shows that very simple crossings can still result in fluctuations that converge arbitrarily slowly, demonstrating in principle the necessity of explicit quantitative information on the tightness of the sequences in order to say something meaningful about the fluctuations.

\begin{example}
\label{ex:slowfluc}
Let $\seq{a_n}\subset [0,1]$ be any decreasing sequence of reals that converges to zero, and define $\seq{X_n}$ by
\begin{equation*}
X_n:=nI_{[0,a_n)}.
\end{equation*}
Then for any $\alpha<\beta$, the sequence $\seq{X_n(\omega)}$ has at most two $[\alpha,\beta]$-crossings, and so $\crsinf{[\alpha,\beta]}{X_n}\leq 2$ almost surely, or alternatively, $\phi_{\alpha,\beta}(\lambda):=3$ is a modulus of $[\alpha,\beta]$-crossings. On the other hand, $\flucinf{1}{X_n(\omega)}\geq N$ if and only if $X_n(\omega)=n$ for all $n\leq N$, or in other words $\omega\in [0,a_N)$, and therefore
\begin{equation*}
\PP\left(\flucinf{1}{X_n}\geq N\right)\leq a_n
\end{equation*}
Therefore, even restricting $\seq{X_n}$ to the class of stochastic processes with crossings bounded above by $2$ does not tell us anything about the convergence speed of 
\begin{equation*}
\PP\left(\flucinf{1}{X_n}\geq N\right)\to 0
\end{equation*}
\end{example}

%%%%%%%%%%%%%%%%%%%%%%%%%%%%%%%%%%%%%%%%%%%%%%%%%%%%%%%%%%%%%%%%%%%%%%%%%%%%%%%%%%%%%%%%%%
%%%%%%%%%%%%%%%%%%%%%%%%%%%%%%%%%%%%%%%%%%%%%%%%%%%%%%%%%%%%%%%%%%%%%%%%%%%%%%%%%%%%%%%%%%
\section{From finite crossings to uniform convergence}
\label{sec:uniform}
%%%%%%%%%%%%%%%%%%%%%%%%%%%%%%%%%%%%%%%%%%%%%%%%%%%%%%%%%%%%%%%%%%%%%%%%%%%%%%%%%%%%%%%%%%
%%%%%%%%%%%%%%%%%%%%%%%%%%%%%%%%%%%%%%%%%%%%%%%%%%%%%%%%%%%%%%%%%%%%%%%%%%%%%%%%%%%%%%%%%%

Having shown in the previous section how to obtain learnable rates of pointwise convergence from a modulus of crossings, we now give an analogous result for uniform convergence. This will then allow us to give a direct quantitative analogue of Doob's martingale convergence theorem, along with more general convergence results for which upcrossing inequalities can be found. Upcrossing inequalities (including Doob's famous inequalities for martingales) typically come in the form of a bound on $\EE[\upcrinf{[\alpha,\beta]}{X_n}]$ for any $\alpha<\beta$. We start by giving a formulation of this property similar in spirit to the modulus of finite crossings introduced in Section \ref{sec:finitizing}.

\begin{definition}
    Any function $\psi:(0,\infty)\times \NN\to\RR$ satisfying
    \[
    \EE\left[\crsinf{[\alpha,\beta]}{X_n}\right]<\psi(M,l)
    \]
    for all $M,l$ and $[\alpha,\beta]\in \mathcal{P}(M,l)$ is called a modulus of $L_1$-crossing for $\seq{X_n}$.
\end{definition}

Our first step is to provide another analogue of Proposition \ref{def:P:subintervals}, using arguments similar to \cite[Section 4]{avigad-gerhardy-towsner:10:local}, though our presentation is more abstract and the issue of $L_1$-boundedness (not considered directly for upcrossings in \cite{avigad-gerhardy-towsner:10:local}) is for now handled by relativising via boundedness as an event.

\begin{proposition}
\label{res:fluctuations:crossings:unif}
Let $\seq{X_n}$ be a stochastic process with modulus of $L_1$-crossings $\psi$ and let $M>0$. Then the formula
\begin{equation*}
Q_M(\varepsilon,n,m):=(\exists l,k\in [n;m](|X_l-X_k|\geq \varepsilon))\cap (|X_n|\leq M)
\end{equation*}
satisfies
\[
\sum_{i=0}^\infty \PP(Q_M(\varepsilon,a_i,b_i))\leq \omega_M(\varepsilon)
\]
uniformly in $a_0<b_0\leq a_1<b_1\leq \ldots$ where
\begin{equation*}
\omega_M(\varepsilon):=(p+2)\cdot\psi\left(M\left(1+\frac{2}{p}\right),p+2\right) \ \ \ \mbox{for } p:=\Bigl\lceil \frac{8M}{\varepsilon}\Bigr\rceil.
\end{equation*}
\end{proposition}

\begin{proof}
Fix $\varepsilon>0$ and $a_0<b_0\leq a_1<b_1\leq \ldots$ and define the formula $A_i$ and $B_i$ for $i\in\NN$ by
\begin{equation*}
\begin{aligned}
A_i:=\exists k,l\in [a_i;b_i](|X_k-X_l|\geq \varepsilon) \ \ \ \mbox{and} \ \ \  B_i:=|X_{a_i}|\leq M.
\end{aligned}
\end{equation*}
Divide $[-M,M]$ into $p=\lceil 8M/\varepsilon\rceil$ equal subintervals, which we label $[\alpha_j,\beta_j]$ for $j=1,\ldots,p$, and add two further intervals of the same width on either side of $[-M,M]$, which we also label $[\alpha_j,\beta_j]$ for $j=0$ and $j=p+1$. In other words
\begin{equation*}
\{[\alpha_0,\beta_0],\ldots,[\alpha_{p+1},\beta_{p+1}]\}=\mathcal{P}(M(1+2/p),p+2).
\end{equation*}
These intervals must have width $\leq \varepsilon/4$. Suppose that $\omega\in A_i\cap B_i$, so that there exists $k(\omega),l(\omega)\in [a_i;b_i]$ with $|X_{k(\omega)}(\omega)-X_{l(\omega)}(\omega)|\geq \varepsilon$, and also $|X_{a_i}(\omega)|\leq M$.

Then by the triangle inequality, either $|X_{a_i}(\omega)-X_{k(\omega)}(\omega)|\geq \varepsilon/2$ or $|X_{a_i}(\omega)-X_{l(\omega)}(\omega)|\geq \varepsilon/2$. Since $X_{a_i}(\omega)\in [-M,M]$ and we have the additional intervals $[\alpha_0,\beta_0]$ and $[\alpha_{p+1},\beta_{p+1}]$, it follows that one of the intervals $[\alpha_j,\beta_j]$ for $j=0,\ldots,p+1$ is crossed by $\seq{X_n(\omega)}$ somewhere in $[a_i;b_i]$, and therefore defining
\begin{equation*}
T_{i,j}:=\mbox{$\seq{X_n}$ crosses $[\alpha_j,\beta_j]$ somewhere in $[a_i;b_i]$}
\end{equation*}
we have shown that
\begin{equation*}
A_i\cap B_i\subseteq \bigcup_{j=0}^{p+1} T_{i,j}.
\end{equation*}
Now suppose that $r<\sum_{i=0}^\infty \PP(A_i\cap B_i)$ for some $r>0$, which in particular means that for some $N\in\NN$ we have
\begin{equation*}
r<\sum_{i=0}^N \PP(A_i\cap B_i)\leq \sum_{i=0}^N \PP\left(\bigcup_{j=0}^{p+1} T_{i,j}\right)\leq \sum_{i=0}^N\sum_{j=0}^{p+1} \PP(T_{i,j})
\end{equation*}
and so there is some $j\in \{0,\ldots,p+1\}$ such that
\begin{equation*}
\begin{aligned}
&\frac{r}{p+2}<\sum_{i=0}^N \PP(T_{i,j})\leq \sum_{i=0}^\infty \PP(T_{i,j})=\sum_{i=0}^\infty \EE\left[I_{T_{i,j}}\right]\\
&=\EE\left[\sum_{i=0}^\infty I_{T_{i,j}}\right]\leq \EE\left[\crsinf{[\alpha_j,\beta_j]}{X_n}\right]\leq\psi\left(M\left(1+\frac{2}{p}\right),p+2\right)
\end{aligned}
\end{equation*}
and therefore
\begin{equation*}
\sum_{i=0}^\infty \PP(Q_M(\varepsilon,a_i,b_i))=\sum_{i=0}^\infty \PP(A_i\cap B_i)\leq (p+2)\cdot \psi\left(M\left(1+\frac{2}{p}\right),p+2\right)
\end{equation*}
and the result follows.
\end{proof}

We can now present our main result on uniform metastability for stochastic processes, in which integrability is now replaced by the more general assumption that $\seq{X_n}$ has a modulus of tightness.

\begin{theorem}
\label{res:crossings:metastability}
Let $\seq{X_n}$ be a stochastic process with modulus of $L_1$-crossings $\psi$. Let $\omega_M(\varepsilon)$ be defined in terms of $\psi$ as in Proposition \ref{res:fluctuations:crossings:unif} i.e.
\begin{equation*}
\omega_M(\varepsilon):=(p+2)\cdot\psi\left(M\left(1+\frac{2}{p}\right),p+2\right) \ \ \ \mbox{for } p:=\Bigl\lceil \frac{8M}{\varepsilon}\Bigr\rceil.
\end{equation*}
If $\seq{X_n}$ has modulus of tightness $h(\lambda)$ then $\seq{X_n}$ converges almost surely with learnable rate of uniform convergence given by:
	\begin{equation*}
	\phi(\lambda,\varepsilon):=\frac{2\omega_{h(\lambda/2)}(\varepsilon)}{\lambda}.
	\end{equation*}

\end{theorem}

\begin{proof}
We define $Q_M(\varepsilon,n,m)$ as in the proof of Proposition \ref{res:fluctuations:crossings:unif}. Fix $\lambda>0$ and define $M_\lambda:=h(\lambda/2)$. By Proposition \ref{res:fluctuations:crossings:unif}, for any $\varepsilon>0$ and $a_0<b_0\leq a_1<b_1\leq\ldots$ we have
	\begin{equation*}
	\sum_{i=0}^\infty \PP(Q_{M_\lambda}(\varepsilon,a_i,b_i))\leq \omega_{M_\lambda}(\varepsilon)
	\end{equation*}
    and so there exists some $n\leq 2\omega_{M_\lambda}(\varepsilon)/\lambda$ such that $Q_{M_\lambda}(\varepsilon,a_n,b_n)<\lambda/2$ i.e.
    \[
    \PP(\exists l,k\in [a_n;b_n](|X_l-X_k|\geq \varepsilon)\cap |X_{a_n}|\leq M_\lambda)<\frac{\lambda}{2}
    \]
	But then it follows that
	\begin{equation*}
	\PP(\exists k,l\in [a_n;b_n](|X_k-X_l|\geq \varepsilon)+\PP(|X_{a_n}|\leq M_\lambda)-1<\frac{\lambda}{2}
	\end{equation*}
	and therefore
	\begin{equation*}
	\begin{aligned}
	\PP(\exists k,l\in [a_n;b_n](|X_k-X_l|\geq \varepsilon)&<\frac{\lambda}{2}-\PP(|X_{a_n}|\leq M_\lambda)+1\\
	&=\frac{\lambda}{2}+\PP(|X_{a_n}|>M_\lambda)<\lambda
	\end{aligned}
	\end{equation*}
 which completes the proof.
\end{proof}

Before moving on to applications, we note that all of our abstract quantitative results can be reformulated in a slightly more precise sense by replacing all assumptions about the existence of moduli that work globally with the exact information that we need, which can then be more easily characterised as a `finitary' convergence principle in the sense of Tao \cite{tao:07:softanalysis}. This does not require any further work and can be read directly from the proofs. As an example, which will help make the discussion in the following section more transparent, we give such a reformulated version of Theorem \ref{res:crossings:metastability} above:

\begin{theorem}
    \label{res:crossings:metastability:reform}
    Let $\seq{X_n}$ be a stochastic process and fix $\lambda,\varepsilon>0$. Then whenever $L$ and $M$ satisfy
        \[
    \PP\left(|X_n|\geq M\right)<\frac{\lambda}{2} \ \ \ \mbox{for all $n\in\NN$}
    \]
    and for $p:=\lceil 8M/\varepsilon\rceil$
    \[
    \EE\left[\crsinf{[\alpha,\beta]}{X_n}\right]\leq L \ \ \ \mbox{for all }[\alpha,\beta]\in\mathcal{P}\left(M\left(1+\frac{2}{p}\right),p+2\right) 
    \]
    then for any $a_0<b_0\leq a_1<b_1\leq \ldots$ we have
    \[
    \PP\left(\exists k,l\in [a_n;b_n]\, (|X_k-X_l|\geq \varepsilon)\right)<\lambda \ \ \ \mbox{for some }n\leq \frac{2(p+2)L}{\lambda}.
    \]
\end{theorem}
We then reobtain Theorem \ref{res:crossings:metastability} from Theorem \ref{res:crossings:metastability:reform} by making the substitutions 
\[
L:=\psi\left(M\left(1+\frac{2}{p}\right),p+2\right) \ \ \ \mbox{and} \ \ \ M:=f\left(\frac{\lambda}{2}\right)
\]

%%%%%%%%%%%%%%%%%%%%%%%%%%%%%%%%%%%%%%%%%%%%%%%%%%%%%%%%%%%%%%%%%%%%%%%%%%%%%%%%%%%%%%%%%%
%%%%%%%%%%%%%%%%%%%%%%%%%%%%%%%%%%%%%%%%%%%%%%%%%%%%%%%%%%%%%%%%%%%%%%%%%%%%%%%%%%%%%%%%%%
\section{Applications}
\label{sec:applications}
%%%%%%%%%%%%%%%%%%%%%%%%%%%%%%%%%%%%%%%%%%%%%%%%%%%%%%%%%%%%%%%%%%%%%%%%%%%%%%%%%%%%%%%%%%
%%%%%%%%%%%%%%%%%%%%%%%%%%%%%%%%%%%%%%%%%%%%%%%%%%%%%%%%%%%%%%%%%%%%%%%%%%%%%%%%%%%%%%%%%%

We now present a series of examples demonstrating how the quantitative analysis of crossing inequalities can be used to obtain concrete metastable rates of uniform convergence for stochastic processes, each following from the Theorem \ref{res:crossings:metastability} in combination with known crossing inequalities. In everything that follows we freely use the obvious fact that if $\phi_1(\lambda,\varepsilon)$ is a learnable rate of uniform convergence and $\phi_1(\lambda,\varepsilon)\leq \phi_2(\lambda,\varepsilon)$ for all $\lambda,\varepsilon>0$, then $\phi_2(\lambda,\varepsilon)$ is also a learnable rate of uniform convergence. This allows us to simplify our rates throughout. First, we note a very simple yet useful instance Theorem \ref{res:crossings:metastability}:
\begin{theorem}
\label{res:unif:monotone}
Suppose that the stochastic process $\seq{X_n}$ is almost surely monotone and has a modulus of tightness $h(\lambda) \in [1,\infty)$, for all $\lambda > 0$. Then $\seq{X_n}$ has a learnable rate of uniform convergence given by:
\begin{equation*}
\phi(\lambda,\varepsilon):=\frac{c}{\lambda\varepsilon}\cdot h\left(\frac{\lambda}{2}\right)
\end{equation*}
for a universal constant $c\leq 22$. In the special case that $\sup_{n\in\NN}\norm{X_n}_\infty<K$, for some $K \ge 1$, the rate becomes
\begin{equation*}
\phi(\lambda,\varepsilon):=\frac{cK}{\lambda\varepsilon}.
\end{equation*}
\end{theorem}

\begin{proof}
Since $\seq{X_n(\omega)}$ is almost surely monotone, we have $\EE(\crsinf{[\alpha,\beta]}{X_n})\leq 1$ for any $\alpha<\beta$, and thus a modulus of uniformly bounded crossings is given by the constant function $\psi(M,l)=1$. The result follows from Theorem \ref{res:crossings:metastability}, noting that in this case
\begin{equation*}
\omega_M(\varepsilon)=\Bigl\lceil \frac{8M}{\varepsilon}\Bigr\rceil+2\leq \frac{11M}{\varepsilon}
\end{equation*}
and therefore
\[
\frac{2\omega_{h(\lambda/2)}(\varepsilon)}{\lambda}\leq \frac{22}{\lambda\varepsilon}\cdot h\left(\frac{\lambda}{2}\right)
\]
and we are done.
\end{proof}
Example \ref{ex:uniform:stronger:pointwise} shows that, particularly for the special case $\sup_{n\in\NN}\norm{X_n}_\infty<K$, the bound on Theorem \ref{res:unif:monotone} above is optimal. We now show how our framework applies to a much more interesting class of stochastic processes.

%%%%%%%%%%%%%%%%%%%%%%%%%%%%%%%%%%%%%%%%%%%%%%%%%%%%%%%%%%%%%%%%%%%%%%%%%%%%%%%%%%%%%%%%%%
\subsection{Martingale convergence}
\label{sec:applications:doob}
%%%%%%%%%%%%%%%%%%%%%%%%%%%%%%%%%%%%%%%%%%%%%%%%%%%%%%%%%%%%%%%%%%%%%%%%%%%%%%%%%%%%%%%%%%

The most well-known crossing inequalities apply to various classes of martingales. We are therefore able to provide learnable rates of uniform convergence in such cases. Obtaining basic uniform rates of metastability for martingales follows in a straightforward way from Theorem \ref{res:crossings:metastability} (which, as we have mentioned, adapts the argument given in \cite[Section 4]{avigad-gerhardy-towsner:10:local} to the $L_1$-setting). However, optimizing these rates is more subtle, and here, a careful choice of the relevant upcrossing inequality is crucial. Because quantitative martingale convergence theorems are an important stepping stone for any future applications of proof mining on stochastic processes, we take particular care in providing best-possible rates. Despite the large body of work done on the computability theory of martingales \cite{hoyrup-rute:21:algorithmic:randomness,rute:etal:algorithmic:doob}, to the best of the authors' knowledge, explicit metastable rates for martingales have never been given, and more importantly, our framework extends to a much broader class of stochastic processes, including almost-supermartingales, as illustrated in Section \ref{sec:applications:almost}.

Let $\seq{\mathcal{F}_n}$ be a filtration with respect to $(\Omega,\mathcal{F},\PP)$ i.e. $\mathcal{F}_0\subseteq \mathcal{F}_1\subseteq \ldots\subseteq \mathcal{F}$ and suppose that $\seq{X_n}$ is adapted to $\seq{\mathcal{F}_n}$. Recall that $\seq{X_n}$ a martingale if for all $n\in\NN$:
\begin{enumerate}[(i)]

	\item $\EE(|X_n|)<\infty$\smallskip
	
	\item $\EE[X_{n+1}\, | \, \mathcal{F}_n]= X_n$ almost surely.\smallskip

\end{enumerate}
It is a submartingale (respectively supermartingale) if the equality in condition (ii) above is weakened to $\geq$ (respectively $\leq$). Doob's classic martingale convergence theorems \cite{doob:53:stochastic} states that $L_1$-bounded sub- or supermartingales converge almost surely to a random variable that is finite almost surely. Our first main result below is a quantitative version of this theorem in the special case of nonnegative submartingales, which we can then extend to obtain rates of the same complexity (up to a constant) for general sub- or supermartingales by applying various decomposition theorems. Furthermore, we incorporate an improvement that holds when we can assume stronger boundedness conditions on the $p$th moment. In what follows, for a random variable $X$, write $\norm{X}_p:= \EE(|X|^p)^{1/p}$ for the standard $L_p$ norm.

\begin{theorem}[Quantitative positive submartingale convergence theorem]
\label{res:unif:doob}
Let $\seq{X_n}$ be a nonnegative submartingale, $p\in [1,\infty]$, and suppose that $K\ge1$ is such that
\begin{equation*}
\sup_{n\in\NN}\norm{X_n}_p< K.
\end{equation*}
Then $\seq{X_n}$ has learnable rate of uniform convergence given by:
\begin{equation*}
 \phi_p(\lambda,\varepsilon):=\frac{cK^2}{\lambda \varepsilon^2}\cdot \left(\frac{2}{\lambda}\right)^{1/p}
\end{equation*}
for $p\in [1,\infty)$ and in the case $p=\infty$ we have a rate given by
\begin{equation*}
\phi_\infty(\lambda,\varepsilon):=\frac{cK^2}{\lambda \varepsilon^2}
\end{equation*}
for a universal constant $c\leq 220$ (and independent of $p$).
\end{theorem}

\begin{proof}
Let $\alpha<\beta$. Doob's well-known upcrossing inequality (see, e.g.\ \cite{doob1961notes}) states that for all $N \in \NN$,
\begin{equation*}
\EE(\upcr{N}{[\alpha,\beta]}{X_n})\leq \frac{\EE((X_N - \alpha)^+)}{\beta-\alpha}.
\end{equation*}
Now since $\seq{X_n}$ is nonnegative, $\upcr{N}{[\alpha,\beta]}{X_n} = 0$ if $\alpha < 0$, and if $\alpha\geq 0$ then $(X_N-\alpha)^+\leq X_N$ and thus
\begin{equation*}
\EE(\upcr{N}{[\alpha,\beta]}{X_n})\le \frac{\EE(X_N)}{\beta-\alpha} \le \frac{\sup_{n\in\NN}\EE(|X_n|)}{\beta-\alpha}\leq \frac{\sup_{n\in\NN}\norm{X_n}_p}{\beta-\alpha}<\frac{K}{\beta-\alpha}
\end{equation*}
which implies 
\begin{equation*}
\EE(\upcrinf{[\alpha,\beta]}{X_n})\leq \frac{K}{\beta-\alpha}.
\end{equation*}
Referring to Remark \ref{rem:upcr} this implies that
\begin{equation*}
\EE(\crsinf{[\alpha,\beta]}{X_n})\leq \frac{2K}{\beta-\alpha} + 1.
\end{equation*}
Now observing that for any $[\alpha,\beta]\in \mathcal{P}(M,l)$ we have $\beta-\alpha=2M/l$, it follows that
\begin{equation*}
\EE(\crsinf{[\alpha,\beta]}{X_n})\leq \frac{lK}{M} + 1
\end{equation*}
and thus $\psi(M,l):=lK/M+1$ is a modulus of crossings for $\seq{X_n}$. Now if $p\in [1,\infty)$ then
\begin{equation*}
\PP(|X_n|\geq N)\leq \frac{\EE(|X_n|^p)}{N^p}< \left(\frac{K}{N}\right)^p
\end{equation*}
and so $h(\lambda):=K\lambda^{-1/p}$ is a modulus of tightness for $\seq{X_n}$. Applying Theorem \ref{res:crossings:metastability} (b), setting $q:=\lceil 8M/\varepsilon\rceil \le 9M/\varepsilon$ and assuming for now that $M \ge 1$, we have
\begin{equation*}
\omega_M(\varepsilon)= \frac{(q+2)^2K}{M(1+2/q)} + (q+2)\leq \frac{11\cdot 9\cdot MK}{\varepsilon^2} + \frac{11M}{\varepsilon} \le \frac{11\cdot 10\cdot MK}{\varepsilon^2}
\end{equation*}
where we use that $K \ge 1$ and $\varepsilon < 1$ to get the last inequality. Instantiating $M:=h(\lambda/2)=K(2/\lambda)^{1/p}\geq 1$ we have,
\begin{equation*}
\frac{2\omega_{h(\lambda/2)}(\varepsilon)}{\lambda}= \frac{2\cdot 11\cdot 10K^2}{\lambda\varepsilon^2}\left(\frac{2}{\lambda}\right)^{1/p}
\end{equation*}
 from which the main part of the result follows. On the other hand, for $p=\infty$ we have
\begin{equation*}
\PP(|X_n|\geq K)=0
\end{equation*}
and thus $h(\lambda):=K$ is a modulus of tightness, and so the above calculations can be simplified to give the stated rate in that case.
\end{proof}

Doob's upcrossing inequalities hold for general (not necessarily positive) sub- or supermartingales, and so one could adapt the proof of Theorem \ref{res:unif:doob} directly to obtain learnable rates in those cases. However, in the nonnegative case, we would have to handle upcrossings where $\alpha<0$, in which we would only have $(X_N-\alpha)^+\leq |X_N|+|\alpha|$ and accordingly
\[
\EE(\crsinf{[\alpha,\beta]}{X_n})\leq \frac{2(K+|\alpha|)}{\beta-\alpha}+1\leq \frac{l(K+M)}{M}+1
\]
which due to the fact that $M$ dominates $K$ in the subsequent calculation results in a worse bound of $(cK^2/\lambda\varepsilon^2)(2/\lambda)^{2/p}$. This somewhat superficial complication can be circumvented in one of two ways: We can either prove an alternative crossing inequality (as we do for positive almost-supermartingales in Remark \ref{rem:downcrossing}, where symmetric downcrossing inequalities work well), or we can make use of standard decomposition theorems for martingales, exploiting the fact that learnable rates compose well for sums of stochastic processes as follows:

\begin{lemma}
\label{res:sum}
Let $\seq{X_n}$ and $\seq{Y_n}$ be stochastic processes with learnable rates of uniform convergence $\phi_1$ and $\phi_2$ respectively. Then $\seq{X_n+Y_n}$ has a learnable rate of uniform convergence given by,
\begin{equation*}
\phi(\lambda,\varepsilon):= \phi_1(\lambda/2,\varepsilon/2) + \phi_2(\lambda/2,\varepsilon/2)
\end{equation*}
\end{lemma}

\begin{proof}
    Fixing $\varepsilon, \lambda \in (0,1)$ and $a_0<b_0\leq a_1<b_1\leq \ldots$, suppose for contradiction that for all $n\le \phi(\lambda,\varepsilon)$ we have
\begin{equation*}
    \label{eq:sum0}
    \PP(\exists i,j\in [a_n;b_n](|X_i+Y_i-X_j-Y_j|\geq \varepsilon))\geq \lambda.
\end{equation*}
For any $\omega\in\Omega$, if there exists $i(\omega),j(\omega)\in [a_n;b_n]$ such that $|X_{i(\omega)}(\omega)+Y_{i(\omega)}-X_{j(\omega)}-Y_{j(\omega)}|\geq \varepsilon$, by the triangle inequality we must have either $|X_{i(\omega)}-X_{j(\omega)}|\geq \varepsilon/2$ or $|Y_{i(\omega)}-Y_{j(\omega)}|\geq \varepsilon/2$. In other words, for each $n\leq \phi(\lambda,\varepsilon)$, we have
\begin{equation*}
\begin{aligned}
    \lambda&\leq \PP(\exists i,j\in [a_n;b_n](|X_i-X_j|\geq \varepsilon/2)\cup \exists i,j\in [a_n;b_n](|Y_i-Y_j|\geq \varepsilon/2))\\
    &\leq \PP(\exists i,j\in [a_n;b_n](|X_i-X_j|\geq \varepsilon/2))+\PP( \exists i,j\in [a_n;b_n](|Y_i-Y_j|\geq \varepsilon/2))
\end{aligned}
\end{equation*}
and so again, by the triangle inequality, we have either
\[
\PP(\exists i,j\in [a_n;b_n](|X_i-X_j|\geq \varepsilon/2))\geq \lambda/2
\]
or
\[
\PP(\exists i,j\in [a_n;b_n](|Y_i-Y_j|\geq \varepsilon/2))\geq \lambda/2
\]
for each $n\leq \phi(\lambda,\varepsilon)$. But then it follows that either there exists a subsequence $a_{n_0}<b_{n_0}\leq a_{n_1}<b_{n_1}\leq \ldots $ such that
\[
\forall k\leq \phi_1(\lambda/2,\varepsilon/2)\left(\PP(\exists i,j\in [a_{n_k};b_{n_k}](|X_i-X_j|\geq \varepsilon/2))\geq \lambda/2\right)
\]
or a subsequence $a_{m_0}<b_{m_0}\leq a_{m_1}<b_{m_1}\leq \ldots $ such that
\[
\forall k\leq \phi_2(\lambda/2,\varepsilon/2)\left(\PP(\exists i,j\in [a_{m_k};b_{m_k}](|Y_i-Y_j|\geq \varepsilon/2))\geq \lambda/2\right)
\]
which contradict the defining property of $\phi_1$ and $\phi_2$ respectively.
\end{proof}

\begin{theorem}[Quantitative Doob's convergence theorem]
\label{res:unif:subdoob}
Let $\seq{X_n}$ be a sub- or super martingale and suppose that $K\ge1$ is such that
\begin{equation*}
\sup_{n\in\NN}\EE(|X_n|)< K.
\end{equation*}
Then $\seq{X_n}$ has learnable rate of uniform convergence given by:
\begin{equation*}
 \phi(\lambda,\varepsilon):=c\left(\frac{K}{\lambda \varepsilon}\right)^2
\end{equation*}
for a universal constants $c\leq 2^{11}\cdot 3^2\cdot c_1$ for $c_1>0$ as in Theorem \ref{res:unif:doob}.
\end{theorem}

\begin{proof}
    We can assume w.l.o.g. that $\seq{X_n}$ is a submartingale, since if $\seq{X_n}$ is a supermartingale, then $\seq{-X_n}$ is a submartingale which must converge with the same learnable rate. Let $X_n=M_n+A_n$ be the usual Doob decomposition in this case, i.e.
    \[
    M_n:=X_0+\sum_{i=1}^n\left(X_i-\EE[X_i\mid\mathcal{F}_{i-1}]\right) \ \ \ \mbox{and} \ \ \ A_n:=\sum_{i=1}^n\left(\EE[X_i\mid\mathcal{F}_{i-1}]-X_{i-1}\right).
    \]
		where it is easy to show that $\seq{M_n}$ is a martingale and $\seq{A_n}$ is almost surely nonnegative and increasing. Then by Lemma \ref{res:sum} $\seq{X_n}$ has learnable rate of uniform convergence $\phi_1(\lambda/2,\varepsilon/2)+\phi_2(\lambda/2,\varepsilon/2)$ where $\phi_1$ and $\phi_2$ are learnable rates for $\seq{M_n}$ and $\seq{A_n}$ respectively. 
		
		Since $\seq{A_n}$ is almost surely monotone and $\EE(|A_n|)=\EE(A_n)=\EE(X_n-X_0)\leq \EE(|X_n|) + \EE(|X_0|)<2K$ and so has modulus of tightness $2K/\lambda$, by Theorem \ref{res:unif:monotone}, we can define
\[
\phi_2(\lambda,\varepsilon):=\frac{4c_2K}{\lambda^2\varepsilon}.
\]
for some $c_2\leq 22$. On the other hand, since $\seq{M_n}$ is a martingale and $\EE(|M_n|)=\EE(|X_n-A_n|)\leq\EE(|X_n|) + \EE(|A_n|)< 3K$ for all $n\in\NN$, we can write $M_n=M^+_n-M^-_n$ where $\seq{M^+_n}$ and $\seq{M^-_n}$ are both nonnegative submartingales (since $x\mapsto x^+$ and $x\mapsto x^-$ are convex functions, and any convex function applied to a martingale results in a submartingale), and have means uniformly bounded by $3K$. So, by Lemma \ref{res:sum} and Theorem \ref{res:unif:doob}, noting that a learnable rate of uniform convergence for $\seq{M^-_n}$ is also one for $\seq{-M^-_n}$, we can define, for $c_1\leq 220$,
\[
\phi_1(\lambda,\varepsilon):=2\left(\frac{c_1(3K)^2}{(\lambda/2)(\varepsilon/2)^2}\left(\frac{4}{\lambda}\right)\right)=\frac{2^6\cdot 3^2\cdot c_1K^2}{\lambda^2\varepsilon^2}
\]
and therefore
\[
\phi_1(\lambda/2,\varepsilon/2)+\phi_2(\lambda/2,\varepsilon/2)\leq \frac{2^{10}\cdot 3^2\cdot c_1K^2}{\lambda^2\varepsilon^2}+\frac{2^5\cdot c_2K}{\lambda^2\varepsilon}\leq \frac{2^{11}\cdot 3^2\cdot c_1K^2}{\lambda^2\varepsilon^2}
\]
assuming $c_2\leq c_1$, and the main result follows directly.   
\end{proof}

\begin{remark}
Though rates of the form $c(K/\lambda\varepsilon)^2$ for general sub- or supermartingales are easiest to establish using decomposition theorems, better constants $c>0$ are likely to be obtained through a optimized up- or downcrossing inequalities, as in Remark \ref{rem:downcrossing}.
\end{remark}

\begin{remark}
\label{rem:doob:variations}
 Sharper results are possible in the case of $L_2$-martingales by appealing instead to known fluctuations bounds. More specifically, for $p\in (1,\infty)$ it is known that if $\sup_{n\in\NN}\norm{X_n}_p\leq K$ then for any $\varepsilon>0$:
\begin{equation*}
\EE\left[(\flucinf{\varepsilon}{X_n})^{p/2}\right]\leq \frac{c_pK^p}{\varepsilon^p}
\end{equation*}
for a constant $c_p>0$ depending only on $p$ (cf. \cite[Theorem 34]{kachurovskii:96:convergence}). For $p\geq 2$, we can then apply Jensen's inequality to obtain 
\begin{equation*}
\EE(\flucinf{\varepsilon}{X_n})\leq \frac{c_p^{2/p} K^2}{\varepsilon^2}
\end{equation*}
and thus by Theorem \ref{res:fluc:to:convergence} (ii) we can improve Theorem \ref{res:unif:subdoob} to
\begin{equation*}
\phi(\lambda,\varepsilon):=\frac{c_p^{2/p} K^2}{\lambda\varepsilon^2}
\end{equation*}
in this special case. However, this simple strategy does not seem possible for general $L_p$ bounded martingales for $p \in [1,2)$: In \cite[Theorem 34]{kachurovskii:96:convergence} it is also shown that there exists an $L_1$ bounded martingale $\seq{X_n}$ such that $\EE(\flucinf{\varepsilon}{X_n}^{1/2})=\infty$, and so in particular $\EE(\flucinf{\varepsilon}{X_n})=\infty$. 
\end{remark}

\begin{remark}
\label{rem:doob:optimality}
    It is natural then to ask whether our learnability bound can be improved in the case of $L_1$-martingales, and here we claim that it is optimal in a certain sense. Any learnable rate of uniform convergence is immediately a learnable rate of pointwise convergence, where the latter is closely connected to fluctuation bounds, and so a sensible approach might be to consider known optimal bounds in the case of fluctuations. A result of Chashka (\cite{chashka:94:fluctuations} but see also \cite[Theorem 28]{kachurovskii:96:convergence}) asserts that if $\seq{X_n}$ is a martingale with $\sup_{n\in\NN}\norm{X_n}_p\leq K$ and $p\in [1,\infty)$ then
\begin{equation}
\label{eqn:chashka}
\PP(\flucinf{\varepsilon}{X_n}\geq N)\leq \frac{C_pK^p}{N^{p/2}\varepsilon^p}
\end{equation}
for some constant $C_p$ dependent only on $p$. The above bound corresponds to the following modulus of finite fluctuations:
\begin{equation}
\label{eqn:optimal}
\phi(\lambda,\varepsilon):=\frac{C_p^{2/p} K^2}{\varepsilon^2}\left(\frac{1}{\lambda}\right)^{2/p}
\end{equation}
and therefore a learnable rate of \emph{pointwise} convergence for $\seq{X_n}$ is also given by $\phi$. However, Chashka shows that the bound (\ref{eqn:chashka}) is optimal in the case of fluctuations, in a result reported in \cite[Theorem 35]{kachurovskii:96:convergence}. Assuming for simplicity that $p=1$, this result is proven by showing that for any $\varepsilon\in (0,1)$ and $k\in\NN$, one can construct an $L_1$-bounded martingale $\seq{X_0,\ldots,X_k}$ with $\EE|X_n|\leq 1$ for all $n$ such that
\[
\PP(\seq{X_0,\ldots,X_k}\mbox{ experiences $k$ $\varepsilon$-fluctuations})\geq \frac{1}{\varepsilon\sqrt{6k}}
\]
and so in particular for any $\varepsilon,\lambda\in (0,1)$ there exists a martingale $\seq{X_n}$ with $\sup_{n\in\NN}|X_n|\leq 1$ such that
\[
\PP\left(\flucinf{\varepsilon}{X_n}\geq \frac{1}{6\lambda^2\varepsilon^2}\right)\geq \lambda
\]
which would, for example, contradict any claim that a modulus of finite fluctuations of type $\psi(\lambda,\varepsilon)=c/\lambda\varepsilon^2$ (as in the $L_2$ case) also works over all $L_1$-martingales. An optimal modulus of finite fluctuations is not automatically an optimal learnable rate of pointwise convergence (though see Remark \ref{rem:chashka} below), however, in Chashka's construction, it is the case that
\[
\PP(\forall n<k\, |X_n-X_{n+1}|\geq \varepsilon)=\PP(\seq{X_0,\ldots,X_k}\mbox{ experiences $k$ $\varepsilon$-fluctuations})
\]
and therefore for any $\varepsilon$ and $k\in\NN$ there exists a martingale $\seq{X_n}$ such that
\[
\PP(\forall n<k\, |X_n-X_{n+1}|\geq \varepsilon)\geq\frac{1}{\varepsilon\sqrt{6k}}
\]
or equivalently, for any $\lambda\in (0,1)$:
\[
\PP\left(\forall n<\frac{1}{6\lambda^2\varepsilon^2}\, |X_n-X_{n+1}|\geq\varepsilon\right)\geq \lambda
\]
Similarly to the case of fluctuations, then, this rules out the existence of a learnable rate of pointwise convergence applicable over all martingales that are asymptotically better than $c/\lambda^2\varepsilon^2$, and since any rate of uniform convergence is automatically a rate of pointwise convergence, this establishes optimality of the bound given in Theorem \ref{res:unif:subdoob}.
\end{remark}

\begin{remark}
\label{rem:chashka}
Our results indicate that rates for finite fluctuations, pointwise and uniform learnability all coincide in the special case of martingales. Some further insight into the close relationship between fluctuations and pointwise learnability can be found in \cite{chashka:94:fluctuations}, where it is implicitly shown that a modulus of finite fluctuations for a martingale follows from what is essentially a learnable rate of pointwise convergence for a stopped version of the martingale (cf. proof of \cite[Theorem 1]{chashka:94:fluctuations}). As such, fluctuations and pointwise learnability are indeed equivalent for martingales, and this construction can potentially be extended to sub- and supermartingales (though not to arbitrary classes of stochastic processes, as it relies crucially on properties of martingales). However, even in the case of martingales, uniform learnability seems to be a distinct concept, with the fact that it coincides with the other notions due to the existence of strong upcrossing inequalities (which are not needed to show the equivalence of pointwise rates with fluctuation bounds along the lines of \cite{chashka:94:fluctuations}). More generally, the relationship between fluctuations and uniform learnability for the classes of stochastic processes considered in this paper raises some open questions that we discuss in more detail in Section \ref{sec:future:fluc}.
\end{remark}

%%%%%%%%%%%%%%%%%%%%%%%%%%%%%%%%%%%%%%%%%%%%%%%%%%%%%%%%%%%%%%%%%%%%%%%%%%%%%%%%%%%%%%%%%%
\subsection{The pointwise ergodic theorem}
\label{sec:applications:pointwise}
%%%%%%%%%%%%%%%%%%%%%%%%%%%%%%%%%%%%%%%%%%%%%%%%%%%%%%%%%%%%%%%%%%%%%%%%%%%%%%%%%%%%%%%%%%
In this short section, we demonstrate how our framework allows us to extract rates of metastability for the Birkhoff ergodic theorem, as done in \cite{avigad-gerhardy-towsner:10:local}. Here the ergodic averages $\seq{A_nf}$ are defined exactly as in Example \ref{ex:ivanov}. 

In \cite{avigad-gerhardy-towsner:10:local}, metastable rates of uniform convergence for $\seq{A_nf}$ are provided for $f\in L_2(X)$ both by mining a proof of Billingsley \cite{billingsley:78:book} and through a simpler analysis of the relevant upcrossing inequality, using an argument which we generalise in Section \ref{sec:uniform}. A comparison of the rates obtained through each approach is given (where, roughly speaking, uniform metastable rates obtained through the proof of Billingsley require asymptotically fewer iterations of a faster growing function). A standard density argument allows this result to be extended to the $f\in L_1(X)$, assuming a sequence of approximating functions for $f$ in $L_2(X)$ with a given rate. Here, we show how our generalised framework allows for a direct argument in the $L_1$ case.

We have the following upcrossing inequality due to Bishop (\cite{bishop:66:upcrossing}, see also \cite[Section 4]{avigad-gerhardy-towsner:10:local}):
\begin{equation*}
\EE(\upcrinf{[\alpha,\beta]}{A_nf})\leq \frac{\EE\left((f - \alpha)^+\right)}{\beta-\alpha}.
\end{equation*}
Assuming $f$ is nonnegative and $f \in L_p(X)$ for some $p \in [1,\infty]$, and argument identical to that in the proof of Theorem \ref{res:unif:doob} yields
\begin{equation*}
\EE(\crsinf{[\alpha,\beta]}{A_nf})< \frac{2K}{\beta-\alpha} + 1
\end{equation*}
for any $K$ satisfying $\sup_{n\in\NN}\norm{A_nf}_p < K$. Thus for positive $f$ we can argue as in the proof of Theorem \ref{res:unif:doob} and obtain the same learnable rate of uniform convergence for $\seq{A_n(f)}$. Furthermore by writing $f = f^+ - f^-$ and noting $A_n(f) = A_n(f^+)- A_n(f^-)$ we can argue as in the proof of Theorem \ref{res:unif:subdoob}, using Lemma \ref{res:sum}, to get the following for general $f$: 
\begin{theorem}[Quantitative pointwise ergodic theorem]
\label{res:unif:ergodic}
Let $p\in [1,\infty]$, and suppose that $K\ge 1$ is such that
\begin{equation*}
\sup_{n\in\NN}\norm{A_nf}_p< K
\end{equation*}
Then $\seq{A_nf}$ has learnable rate of uniform convergence given by:
\begin{equation*}
 \phi_p(\lambda,\varepsilon):=\frac{16cK^2}{\lambda \varepsilon^2}\cdot \left(\frac{4}{\lambda}\right)^{1/p}
\end{equation*}
for $p \in [1,\infty)$, and in the case $p =\infty$ we have a rate given by
\begin{equation*}
\phi_\infty(\lambda,\varepsilon):=\frac{16cK^2}{\lambda \varepsilon^2}
\end{equation*}
for a universal constant $c\leq 220$ (and independent of $p$).
\end{theorem}

 The rate in the case $p = \infty$ in Theorem \ref{res:unif:ergodic} corresponds directly to that given in \cite[Theorem 4.1]{avigad-gerhardy-towsner:10:local}, where our additional constant $c$ is an artefact of our more general approach, which could be streamlined for uniformly bounded stochastic processes.

%%%%%%%%%%%%%%%%%%%%%%%%%%%%%%%%%%%%%%%%%%%%%%%%%%%%%%%%%%%%%%%%%%%%%%%%%%%%%%%%%%%%%%%%%%
\subsection{Almost supermartingales}
\label{sec:applications:almost}
%%%%%%%%%%%%%%%%%%%%%%%%%%%%%%%%%%%%%%%%%%%%%%%%%%%%%%%%%%%%%%%%%%%%%%%%%%%%%%%%%%%%%%%%%%

Our main goal in analysing the quantitative convergence of stochastic processes in such a general way was not merely to provide learnable rates for martingales but to facilitate the extension of such techniques to stochastic optimization, where establishing the convergence of a stochastic algorithm is often reduced to proving the convergence of an almost-supermartingale, i.e.\ a stochastic process with additional parameters that behaves like a supermartingale in the limit. We now give a simple example, demonstrating that our (optimal) learnable rates also extend to supermartingales with error terms, thanks to our ability to combine learnable rates as in Lemma \ref{res:sum}. We anticipate a more thorough study of almost supermartingales in future work.

\begin{theorem}
\label{res:optimization}
Let $\seq{X_n}$ be a nonnegative stochastic process satisfying
\[
\EE[X_{n+1}\mid \mathcal{F}_n]\leq X_n+E_n
\]
where $\seq{E_n}$ is nonnegative stochastic process with
\[
\sum_{i=0}^\infty E_i<a
\]
on $\Omega$, for some $a>0$. Let $K> \EE(X_0)$, for some $K\ge 1$. Then $\seq{X_n}$ converges almost surely, with the following learnable rate of uniform convergence
\[
\psi(\lambda,\varepsilon):=16c\left(\frac{K+2a}{\lambda\varepsilon}\right)^2+4c'\left(\frac{a}{\lambda\varepsilon}\right)
\]
where $c$ and $c'$ are the constants from Theorems \ref{res:unif:subdoob} and \ref{res:unif:monotone} respectively.
\end{theorem}

\begin{proof}
Define 
\[
Y_n:=X_n-\sum_{i=0}^{n-1}E_i.
\]
Observing that
\[
\EE[Y_{n+1}\mid \mathcal{F}_n]=\EE[X_{n+1}\mid\mathcal{F}_n]-\sum_{i=0}^{n}E_i\leq X_n-\sum_{i=0}^{n-1} E_i=Y_n
\]
we see that $\seq{Y_n}$ is a supermartingale. By basic properties of conditional expectations, we have
\[
\EE[X_n]\leq \EE[X_0]+\sum_{i=0}^{n-1}\EE(E_i) < K+a
\]
and therefore
\[
\sup_{n\in\NN}\EE[|Y_n|]<K+2a.
\]
Applying Theorem \ref{res:unif:subdoob} it follows that $\seq{Y_n}$ converges with learnable rate of uniform convergence
\[
\phi_1(\lambda,\varepsilon)=c\left(\frac{K+2a}{\lambda\varepsilon}\right)^2
\]
for suitable constant $c$. Since the stochastic process $\left\{\sum_{i=0}^{n-1}E_i\right\}$ is nondecreasing and uniformly bounded by $a$, the special case of Theorem \ref{res:unif:monotone} yields the following learnable rate of uniform convergence:
\[
\phi_2(\lambda,\varepsilon)=c'\left(\frac{a}{\lambda\varepsilon}\right)
\]
and the given rate of follows by applying Lemma \ref{res:sum} to $Y_n+\sum_{i=0}^{n-1}E_i$.
\end{proof}

\begin{remark}
\label{rem:downcrossing}
    A more involved convergence proof for Theorem \ref{res:optimization} along with a slightly different rate involves a single application of Theorem \ref{res:crossings:metastability} to a downcrossing inequality for $\seq{X_n}$. The latter can be obtained through modifying the standard upcrossing inequality for supermartingales (cf. \cite{williams:91:martingales}): Specifically, given $\alpha<\beta$, define the predictable process $C_n\in \{0,1\}$ by $C_1=1$ iff $X_0>\beta$ and $C_{n+1}=1$ if either
    \begin{itemize}
        \item $C_n=1$ and $X_n\ge \alpha$, or
        \item $C_n=0$ and $X_n>\beta$.
    \end{itemize}
    Let $C_n':= 1-C_n$ be the dual predictable process. We can view $\seq{C_n}$ as a deliberately bad gambling strategy that buys $X_n$ when $X_n>\beta$ and sells it when $X_n<\alpha$, and $\seq{C_n'}$ as the opposite of this strategy. First, analogous to the more common argument for good strategies, it follows that, for all $N \in \NN$, the amount lost by the bad strategy over $\seq{X_0,\ldots,X_N}$ is worse than $\beta-\alpha$ multiplied by the number of downcrossings (plus an error term accounting for the final state $X_N$) i.e.
		\[
		\sum_{i=1}^{N}C_i(X_i-X_{i-1})\leq -(\beta-\alpha)\dcr{N}{[\alpha,\beta]}{X_n}+(X_n-\beta)^{+}
		\]
		and now defining the supermartingale $Y_n$ as in the proof of Theorem \ref{res:optimization}, we have
    \[
    \begin{aligned}
    (\beta-\alpha)\dcr{N}{[\alpha,\beta]}{X_n}&\leq -\sum_{i=1}^{N}C_i(X_i-X_{i-1}) + (X_N-\beta)^{+}\\
    &\leq-\sum_{i=1}^{N} C_i(Y_i-Y_{i-1})-\sum_{i=1}^{N}C_iE_{i-1} + (X_N-\beta)^{+}\\
    &\leq-\sum_{i=1}^{N} C_i(Y_i-Y_{i-1}) + (X_N-\beta)^{+}
    \end{aligned}
    \]
    where the last inequality follows from the nonnegativity of $\seq{E_n}$. Since $\seq{C'_n}$ is a bounded, nonnegative predictable process and $\seq{Y_n}$ is a supermartingale, it is a standard result that $\sum_{i=1}^N C'_i(Y_i-Y_{i-1})$ is a supermartingale null at zero. Therefore
    \begin{equation*}
        \EE\left(\sum_{i=1}^{N} C_i'(Y_i-Y_{i-1})\right) \le 0
    \end{equation*}
     and we have,
    \begin{equation*}
    \begin{aligned}
        \EE\left(-\sum_{i=1}^{N} C_i(Y_i-Y_{i-1})\right)&\le  \EE\left(-\sum_{i=1}^{N} C_i(Y_i-Y_{i-1})\right)+\EE\left(-\sum_{i=1}^{N} C_i'(Y_i-Y_{i-1})\right)\\
       & = \EE\left(\sum_{i=1}^{N} Y_{i-1}-Y_i\right) =\EE\left(X_0-X_N+\sum_{i=0}^{N-1} E_i\right)\\
			&\leq \EE(X_0)+\EE\left(\sum_{i=0}^{N-1} E_i\right)<K+a
    \end{aligned} 
    \end{equation*}
    and so it follows that
    \[
    (\beta-\alpha)\EE[\dcr{N}{[\alpha,\beta]}{X_n}]\leq K+a+\EE[(X_N-\beta)^+]
    \]
    Because $\seq{X_n}$ is nonnegative we can assume that $\beta\geq 0$ and therefore
    \[
    \EE[\dcr{N}{[\alpha,\beta]}{X_n}]\leq \frac{K+a+\EE(X_N)}{\beta-\alpha}<\frac{2(K+a)}{\beta-\alpha}
    \]
    and so as in the proof of Theorem \ref{res:unif:subdoob}, 
		\[
		\psi(M,l)=\frac{2l(K+a)}{M}+1
		\]is a modulus of crossings for $\seq{X_n}$. Since $h(\lambda):=(K+a)/\lambda$ is a modulus of tightness for $\seq{X_n}$, and applying Theorem \ref{res:crossings:metastability} yields the following learnable rate of uniform convergence for $\seq{X_n}$:
    \[
    e(\lambda,\varepsilon):=c\left(\frac{K+a}{\lambda\varepsilon}\right)^2
    \]
    for $c\leq 2\cdot 11\cdot 19$, which improves the bound from Theorem \ref{res:optimization} at the cost of additional work. For future applications in stochastic optimization, we predict that optimal bounds will be achieved through making careful decisions on precisely how to apply Theorem \ref{res:crossings:metastability}.
\end{remark}

%%%%%%%%%%%%%%%%%%%%%%%%%%%%%%%%%%%%%%%%%%%%%%%%%%%%%%%%%%%%%%%%%%%%%%%%%%%%%%%%%%%%%%%%%%
\subsection{Summary}
\label{sec:applications:summary}
%%%%%%%%%%%%%%%%%%%%%%%%%%%%%%%%%%%%%%%%%%%%%%%%%%%%%%%%%%%%%%%%%%%%%%%%%%%%%%%%%%%%%%%%%%

We collect together our main results on learnable rates of uniform convergence for martingales in the table below:\medskip

\begin{center}
\begin{tabular}{ c|c|c } 
\hline
 stochastic process & rate & notes \\
 \hline
 constant, monotone & $K/\varepsilon$ & standard in proof mining (cf.\cite{kohlenbach:08:book}) \\ 
 almost sure monotone & $cK/\lambda\varepsilon$ & Theorem \ref{res:unif:monotone}, cf. Example \ref{ex:uniform:stronger:pointwise} \\ 
 $L_2$-martingales & $cK^2/\lambda\varepsilon^2$ & \cite{kachurovskii:96:convergence} and Theorem \ref{res:fluc:to:convergence} (ii) cf. Remark \ref{rem:doob:variations} \\ 
 $L_1$-martingales & $cK^2/\lambda^2\varepsilon^2$ & Theorem \ref{res:unif:subdoob} cf. Remark \ref{rem:doob:optimality} \\ 
 $L_1$-almost-martingales & $cK^2/\lambda^2\varepsilon^2$ & e.g. Theorem \ref{res:optimization}\\\hline
\end{tabular}
\end{center}\medskip

The increasing complexity reflects additional elements needed in the convergence proofs in each case, starting with the basic learnable rate for monotone sequences of real numbers. Almost sure monotone sequences require for the first time a factor of $\lambda$ as there is now a probabilistic component, the increased fluctuation complexity of $L_2$-martingales induces an additional factor of $\varepsilon$, and finally, the weaker integrability assumption in the $L_1$ case along with the reliance on upcrossing inequalities for $[\alpha,\beta]\in [-K/\lambda,K/\lambda]$ is responsible for the final factor of $\lambda$. Most interestingly of all, the almost-supermartingale we consider does not increase the complexity: Here, the addition of error terms contributes less to the overall complexity than the single instance of Doob's theorem required in the convergence proof. We conjecture that this is the case for a much broader class of almost-supermartingales, such as those considered in the landmark paper of Robbins and Siegmund \cite{robbins-siegmund:71:lemma}, raising the prospect of low-complexity bounds even for sophisticated convergence theorems in stochastic optimization.

%%%%%%%%%%%%%%%%%%%%%%%%%%%%%%%%%%%%%%%%%%%%%%%%%%%%%%%%%%%%%%%%%%%%%%%%%%%%%%%%%%%%%%%%%%
%%%%%%%%%%%%%%%%%%%%%%%%%%%%%%%%%%%%%%%%%%%%%%%%%%%%%%%%%%%%%%%%%%%%%%%%%%%%%%%%%%%%%%%%%%
\section{Future work and open questions}
\label{sec:future}
%%%%%%%%%%%%%%%%%%%%%%%%%%%%%%%%%%%%%%%%%%%%%%%%%%%%%%%%%%%%%%%%%%%%%%%%%%%%%%%%%%%%%%%%%%
%%%%%%%%%%%%%%%%%%%%%%%%%%%%%%%%%%%%%%%%%%%%%%%%%%%%%%%%%%%%%%%%%%%%%%%%%%%%%%%%%%%%%%%%%%

In this extended concluding section, we give further justification to our work by outlining, in more detail, the wider research effort to which it belongs (already hinted at in Section \ref{sec:intro:project}) and identifying a range of open questions. We focus, in turn, on: (1) logical aspects of probability - where all our results form new examples of the extractability of uniform rates and hint at new logical phenomena; (2) future applications in stochastic optimization - where in particular our quantitative supermartingale convergence theorem will be crucial for the analysis of more intricate convergence theorems connected to Fej\'er monotonicity in the stochastic setting; and finally (3) the wider literature on quantitative aspects of martingales and ergodic averages - where our work gives rise to several open questions.

%%%%%%%%%%%%%%%%%%%%%%%%%%%%%%%%%%%%%%%%%%%%%%%%%%%%%%%%%%%%%%%%%%%%%%%%%%%%%%%%%%%%%%%%%%
\subsection{Proof-theoretic background}
\label{sec:future:prooftheory}
%%%%%%%%%%%%%%%%%%%%%%%%%%%%%%%%%%%%%%%%%%%%%%%%%%%%%%%%%%%%%%%%%%%%%%%%%%%%%%%%%%%%%%%%%%

We collect together some observations and remarks that attempt to explain, on a logical level, our success in obtaining quantitative information for stochastic convergence theorems. As already noted in Remark \ref{rem:specker}, direct computable rates of convergence for sub- or supermartingales are generally not possible unless one assumes additional computability structure as in \cite{hoyrup-rute:21:algorithmic:randomness,rute:etal:algorithmic:doob}. Our main results demonstrate that, on the other hand, computable metastable rates of convergence can be obtained, and moreover, these are of low complexity and highly uniform, depending only on a bound on the expected values of the stochastic process $\seq{X_n}$.

The extractability of highly uniform rates of metastability (and quantitative information in general) is a well-known phenomenon in proof mining, explained through a series of specialised logical metatheorems (e.g. \cite{kohlenbach-gerhardy:08:metatheorems,kohlenbach-guenzel:16:metatheorems,kohlenbach:05:metatheorems,pischke:24:metatheorems:accretive,pischke:24:metatheorems:dual,paunescu-sipos:23:metatheorems}). These metatheorems expand traditional program extraction theorems, typically G\"odel's Dialectica interpretation \cite{goedel:58:dialectica}, to encompass uniformities in specific mathematical domains, by combining two crucial ingredients:

\begin{enumerate}

	\item \emph{Abstract types} (introduced by Kohlenbach \cite{kohlenbach:05:metatheorems}), which allow one to talk about general spaces without needing them to be explicitly encoded within some arithmetical system;
	
	\item \emph{Majorizability} (introduced by Howard \cite{howard:73:majorizability}), an abstract relation which allows one to express the notion of `bounding information' at all type levels, and in particular on abstract types. 

\end{enumerate}

Very informally, metatheorems arise when one can demonstrate that there exist a collection of axioms for characterising abstract spaces and objects that act on them that are either universal (and thus have no computational content), or can be given explicit majorants. Then, when put through the standard machinery of the Dialectica interpretation in its monotone form, one can show that from any proof that is in principle axiomatisable in this way it is possible extract computational information that depends only on the relevant majorants for those spaces.

The first author and Pischke have recently developed the first such system for probability \cite{neri-pischke:pp:formal}. Here the underlying ground set $\Omega$ of a probability space along with the resulting algebra $S$ of events are both represented as abstract types, with the probability measures and integration, etc. represented as abstract functions of the appropriate types. A majorant for an event $A:S$ is then defined to be a bound on $\PP(A)$. As such, all events are uniformly majorized by $1$, and propagating this through the whole system results in a metatheorem for probability spaces where extracted terms are independent of the underlying events. This then explains, for example, the uniformities of the quantitative Egorov theorem of Avigad et al. (used here in the proof of Theorem \ref{res:equivalent:metastabe} and quoted as Theorem \ref{res:egorov:avigad}) where the functional $\Gamma$ is independent of the events $|X_i-X_j|>\varepsilon$. 

The main results of this paper, though not obtained through the formal application of any logical metatheorems, represent new examples where highly uniform quantitative information was possible. Interestingly, these uniformities go even beyond those explained by \cite{neri-pischke:pp:formal}, where, for example, in Theorem \ref{res:optimization}, the only information we require on the almost-supermartingale $\seq{X_n}$ is a bound on $\EE[X_0]$, whereas the metatheorem presented in \cite{neri-pischke:pp:formal} would have been restricted to the case of bounded random variables. The case studies presented here, therefore, hint at ways in which the initial metatheorem of \cite{neri-pischke:pp:formal} can be augmented with a detailed, abstract treatment of random variables and stochastic processes to capture such phenomena. In this spirit, we make two conjectures:

\begin{enumerate}[(I)]

    \item Uniformities present in both this paper and the many existing quantitative results on martingales can be formally explained by expanding the system from \cite{neri-pischke:pp:formal} with an abstract type for random variables, which will also facilitate the formalisation of stochastic processes. By interpreting different notions of integrability as majorants, one can then explain those uniformities as instances of a general logical metatheorem. For example, taking a majorant of a random variable $X$ to be a bound on $\EE[X]$ would likely cover the main results of this paper, but we anticipate that other bounding notions that play a central role in the theory of martingales, particularly \emph{uniform integrability}, can be elegantly incorporated into the framework of logical metatheorems.\smallskip

    \item Many important theorems in stochastic convergence theory will end up being formalisable within the aforementioned logical systems, and thus amenable to the extraction of uniform rates of convergence or metastability. In particular, we conjecture that the defining properties of martingales, along with associated concepts like conditional expectations, can be elegantly represented and axiomatised with the correct abstract types and operations.\smallskip
    
\end{enumerate}

The expansion of \cite{neri-pischke:pp:formal} to incorporate stochastic processes along with richer notions of majorizability is currently work in progress. 

%%%%%%%%%%%%%%%%%%%%%%%%%%%%%%%%%%%%%%%%%%%%%%%%%%%%%%%%%%%%%%%%%%%%%%%%%%%%%%%%%%%%%%%%%%
\subsection{Almost-supermartingales and stochastic optimization}
\label{sec:future:combettes}
%%%%%%%%%%%%%%%%%%%%%%%%%%%%%%%%%%%%%%%%%%%%%%%%%%%%%%%%%%%%%%%%%%%%%%%%%%%%%%%%%%%%%%%%%%

The main practical purpose of our work has been to bring proof-theoretic methods to bear on stochastic optimization. Current work in progress involves extending our very simple study of almost-supermartingales to the much more powerful convergence theorems used in optimization, such as the Robbins-Siegmund theorem \cite{robbins-siegmund:71:lemma} and variants thereof. The Robbins-Siegmund theorem (an extreme simplification of which is represented by our Theorem \ref{res:optimization}) considers stochastic processes that satisfy the following almost-supermartingale property:
\begin{equation}
\label{eqn:fejer}
\EE[X_{n+1}\mid \mathcal{F}_n]\leq (1+A_n)X_n-B_n+E_n
\end{equation}
where $\seq{A_n}$, $\seq{B_n}$ and $\seq{E_n}$ are nonnegative stochastic processes satisfying
\[
\sum_{i=0}^\infty A_i<\infty \ \ \mbox{and} \ \ \sum_{i=0}^\infty E_i<\infty
\]
almost surely. It asserts that for any such sequences, both $\seq{X_n}$ and $\sum_{i=0}^\infty B_i$ converge almost surely, and these facts are widely used throughout stochastic optimization to establish the convergence of stochastic algorithms. A modern survey on how the Robbins-Siegmund theorem and related results are applied across game theory, convex optimization and machine learning is given in \cite{franci-grammatico:convergence:survey:22}, where (\ref{eqn:fejer}) typically represents some kind of quasi-Fej\'er monotonicity property of an algorithm $\seq{x_n}$ with respect to a set of solutions $x^\ast$ i.e. $X_n:=\norm{x_n-x^\ast}$ for a vector-valued random variable $\seq{x_n}$.

Indeed, once a quantitative convergence result for general almost-supermartingales has been established, the application of almost-supermartingales for Fej\'er monotone sequences represents an obvious next step, as Fej\'er monotonicity in the nonstochastic setting has already been widely explored in proof mining \cite{kohlenbach-leustean-nicolae:18:fejer,kohlenbach-lopezacedo-nicolae:19:fejer,pischke:23:generalized:fejer}. Here, a starting point would be the work of Combettes and Pesquet \cite{combettes-pesquet:15:fejer}, who explicitly connect an abstract notion of stochastic Fej\'er monotonicity in Hilbert spaces with the Robbins-Siegmund theorem. Establishing quantitative results on stochastic Fej\'er monotonicity that encompass \cite{combettes-pesquet:15:fejer} would open the door to numerous applications across different domains where Fej\'er monotonicity is used.

Any progress in this direction relies on a basic quantitative understanding of stochastic processes, along with a useful set of tools for producing rates of metastability for the many kinds of almost-supermartingales that arise in stochastic optimization. This is exactly what we have sought to provide through the present paper.

%%%%%%%%%%%%%%%%%%%%%%%%%%%%%%%%%%%%%%%%%%%%%%%%%%%%%%%%%%%%%%%%%%%%%%%%%%%%%%%%%%%%%%%%%%
\subsection{Fluctuation and upcrossing bounds in the literature}
\label{sec:future:fluc}
%%%%%%%%%%%%%%%%%%%%%%%%%%%%%%%%%%%%%%%%%%%%%%%%%%%%%%%%%%%%%%%%%%%%%%%%%%%%%%%%%%%%%%%%%%

Beyond the two general projects represented outlined in Sections \ref{sec:future:prooftheory} and \ref{sec:future:combettes} above, a number of specific open questions arise directly from our work on the various quantitative notions that surround almost sure convergence, and we list some of them below:\smallskip

\begin{enumerate}

    \item What is the precise relationship between learnable rates of uniform convergence and moduli of finite fluctuations? Both of these immediately result in learnable rates of pointwise convergence, but is there a direct relationship between them? Example \ref{ex:uniform:stronger:pointwise} demonstrates that they do not always coincide for classes of stochastic processes, though they do for $L_1$-martingales (Remark \ref{rem:doob:optimality}) (even though our derivation of the rate of uniform convergence for $L_1$-submartingales seems quite different to that of fluctuations of martingales presented in \cite{chashka:94:fluctuations}). If we were able to prove, for example, that any metastable rate of uniform convergence is also a fluctuation bound (in addition to being a pointwise rate), then an immediate consequence of Theorem \ref{res:unif:ergodic} for $p=1$ would be the existence of a modulus of finite fluctuations of the form
		\[
		\phi(\lambda,\varepsilon):=C\left(\frac{\norm{f}_1}{\lambda^2\varepsilon^2}\right)^2
		\]
		for ergodic averages $\seq{A_nf}$ with $f\in L_1(X)$, a result nontrivial enough that was apparently left as a conjecture in \cite{kachurovskii:96:convergence} (Conjecture 5).
		\smallskip

    \item Can we exploit deeper upcrossing inequalities in the literature to produce interesting quantitative results outside of martingale theory via our abstract Theorems \ref{res:fluctuations:crossings:stochastic} and \ref{res:crossings:metastability}? For example, several such inequalities are given by Hochman in \cite{hochman:09:upcrossing}, whose applications include the Shannon-McMillan-Breiman theorem. In a similar spirit, can any of the fluctuation results or variational inequalities explored in e.g. \cite{jones:etal:08:variational,jones-ostrovskii-rosenblatt:96:square,kalikow1999fluctuations,warren:21:fluctuations} be generalised or brought within our abstract framework? \smallskip

    \item We have shown that in the special case of $L_1$-martingales, our uniform rates coincide with the best known pointwise rates, so in surprising contrast to the main result of Avigad et al.\ \cite{avigad-dean-rute:12:dominated} (see also Theorem \ref{res:egorov:avigad}) there is no increase in complexity going from pointwise to uniform. We conjecture that in many ordinary mathematical situations (for example, the specific class of formulas considered in \ref{res:equivalent:metastabe}), the complexity blowup (if it exists at all) it is significantly better than that currently guaranteed by complexity results on bar recursion. Can we make this more precise and provide an improved bound on the complexity of uniform learnability in terms of pointwise learnability, either in general or in certain restricted situations that one is likely to encounter in practice?\smallskip
    
 %   \item The rate we give in Theorem \ref{res:unif:subdoob} is asymptotically optimal and depends on a universal constant, which we give a crude upper bound for in our proof. Can one find optimal constants for the rates given?\smallskip

\end{enumerate}

\noindent\textbf{Acknowledgements.} The authors are indebted to Nicholas Pischke for numerous insightful discussions on the topics of this paper, along with many valuable comments on an earlier draft of the paper which improved its presentation considerably. The authors also thank Jeremy Avigad for providing several extremely useful pointers to the literature. The first author was partially supported by the EPSRC Centre for Doctoral Training in Digital Entertainment EP/L016540/1, and the second author was partially supported by the EPSRC grant EP/W035847/1.

\bibliographystyle{acm}
\bibliography{tpbiblio}

\end{document}